\documentclass[a4paper,12pt,oneside,dvipsnames]{amsart} 

\usepackage{amsmath,amsfonts,amsthm,amssymb}
\usepackage{xcolor}
\usepackage{graphicx}
\usepackage{cite}
\usepackage{comment}
\usepackage{geometry}
\usepackage{array}
\usepackage{bbm}

\usepackage{hyperref}
\hypersetup{colorlinks=true, citecolor=PineGreen, linkcolor=RoyalBlue, linktoc=page,urlcolor=RoyalBlue}

\theoremstyle{plain}
\newtheorem{theorem}{Theorem}[section] 
\newtheorem{lem}[theorem]{Lemma}
\newtheorem{prop}[theorem]{Proposition}
\newtheorem{cor}[theorem]{Corollary}

\theoremstyle{remark}
\newtheorem{rem}[theorem]{Remark}

\theoremstyle{definition}

\newcommand{\A}{\mathbb{A}}
\newcommand{\Ax}{\mathbb{A}^{\times}}
\newcommand{\Af}{\mathbb{A}_{\textrm{f}}}

\newcommand{\R}{\mathbb{R}}

\newcommand{\C}{\mathbb{C}}
\newcommand{\Cx}{\mathbb{C}^{\times}}
\newcommand{\Hb}{\mathbb{H}}

\newcommand{\Nat}{\mathbb{N}}			
\newcommand{\Z}{\mathbb{Z}}
\newcommand{\Zp}{\mathbb{Z}_{p}}
\newcommand{\Zpx}{\mathbb{Z}_{p}^{\times}}
\newcommand{\Q}{\mathbb{Q}}
\newcommand{\Qx}{\mathbb{Q}^{\times}}
\newcommand{\Qp}{\mathbb{Q}_{p}}

\newcommand{\Pb}{\mathbb{P}}

\newcommand{\Cc}{\mathcal{C}}

\newcommand{\Wc}{\mathcal{W}}

\newcommand{\Wh}{\mathcal{W}}

\newcommand{\af}{\mathfrak{a}}

\newcommand{\GL}{\operatorname{GL}}

\newcommand{\SL}{\operatorname{SL}}

\newcommand{\Orth}{\operatorname{O}}

\newcommand{\n}{\operatorname{n}}
\newcommand{\N}{\operatorname{N}}

\newcommand{\Hom}{\operatorname{Hom}}

\newcommand{\sgn}{\operatorname{sgn}}
\newcommand{\isom}{\cong}
\newcommand{\pr}{\operatorname{pr}}

\newcommand{\Lie}{\operatorname{Lie}}

\newcommand{\f}{\operatorname{f}}
\newcommand{\Ch}{\operatorname{Ch}}
\renewcommand{\mod}{\operatorname{mod}}

\newcommand{\bs}{\backslash}

\newcommand{\Mod}[1]{\ (\operatorname{mod}\ #1)}

\newcommand{\smat}[4]{\left(\begin{smallmatrix}
#1 & #2 \\ #3 & #4
\end{smallmatrix}\right)}

\newcommand{\abs}[1]{\lvert{#1}\rvert}

\newcommand{\emb}{\hookrightarrow}
\newcommand{\emblong}{\ensuremath{\lhook\joinrel\relbar\joinrel\rightarrow}}
\newcommand{\arrup}[2]{\mathrel{\mathop{#1}^{#2}}}

\newcommand{\itemspacing}{
\setlength\topsep{0.1in}
\setlength\itemsep{0.1in}}

\newcommand{\ds}{\displaystyle} 

\newcommand{\tp}[2]{\texorpdfstring{#1}{#2}}

\title{Vorono\"{i} summation via switching cusps}
\date{3$^{\mathrm{rd}}$ April 2019}

\author{Edgar Assing}
\author{Andrew Corbett}

\address{School of Mathematics, University of Bristol, Bristol, UK, BS8 1TW}
\email{edgar.assing@bristol.ac.uk}

\address{Harrison Building, Streatham Campus, North Park Road, Exeter, UK, EX4 4QF}
\email{A.J.Corbett@exeter.ac.uk}

\begin{document}

\begin{abstract}

We consider the Fourier expansion of a Hecke (resp.\ Hecke--Maa\ss) cusp form of general level $N$ at the various cusps of $\Gamma_{0}(N)\bs\Hb$. We explain how to compute these coefficients via the local theory of $p$-adic Whittaker functions and establish a classical Vorono\"i summation formula allowing an arbitrary additive twist. Our discussion has applications to bounding sums of Fourier coefficients and understanding the (generalised) Atkin--Lehner relations.
\end{abstract}

\maketitle


\section{Introduction}

The purpose of this note is to prove two formulae. The first, a Vorono\"i-type summation formula, relating additively twisted Fourier coefficients of a cuspidal Hecke eigenform $f$ to a dual sum of coefficients at another cusp, related to the original twist. Secondly, we give an explicit formula for the Fourier coefficients $f$ at an arbitrary cusp. As applications, we prove upper bounds for sums of Fourier coefficients and we revise the classical Atkin--Lehner relations in a $p$-adic setting.

Whilst similar formulae are, in certain settings, well known, we provide a perspective here that links the two fundamentally and lays bare the mechanics behind additive twists and switching cusps. That perspective is derived from noting that classical Fourier coefficients are special values of $p$-adic Whittaker functions on $\GL_{2}(\Qp)$. In that vein, we intend the following exposition to serve pedagogically as a reference for the analysis of classical Fourier coefficients with adelic machinery.

For the remainder of this introduction we detail our main results\footnote{Our results explicitly handle the case of holomorphic and of Maa\ss\ forms simultaneously.} (Proposition \ref{prop:the_proposition} and Corollary \ref{cor:voronoi} below) for a holomorphic cusp form $f(z)=\sum_{n\geq 1}a_{f}(n)e(nz)$ of level $N\in\Nat$, weight $k\in 2\Nat$ and nebentypus $\chi$.
Fix a modulus $q\in\Nat$ for the additive character $n\mapsto e(an/q)$ where $(a,q)=1$.
Then for a Bruhat--Schwartz function $F\colon\R_{>0}\rightarrow\C$ we prove that
\begin{multline}\label{eq:intro_voronoi}
	\sum_{\n\in\Nat} e\left(\frac{an}{q}\right)a_f(n) n^{-\frac{k-1}{2}} F(n)\\
		=\, q^{-1} \sum_{n\in\Nat}e\left(-\frac{\overline{a}n}{q\delta(\af)}\right)a_f(n,\af)\bigg(\frac{n}{\delta(\af)}\bigg)^{-\frac{k-1}{2}}[\mathcal{H}_f F]\left(\frac{n }{[q^2,Mq,N]}\right)
\end{multline}
where $\mathcal{H}_f F$ is the Hankel transform of $F$ (see \eqref{eq:def_hankel_transform_hol} and \eqref{eq:def_hankel_transform_maa} below) and $a_{f}(n;\af)$ denotes the $n$-th Fourier coefficient of $f$ at the cusp $\af=a/q$ of (extended) cusp width $\delta(\af)$, as defined in \eqref{eq:extended_width}. Formula \eqref{eq:intro_voronoi} constitutes one case of Corollary \ref{cor:voronoi} given later.

Many of the technicalities in evaluating such formulae explicitly, in terms of the coefficients $a_{f}(n)$, are buried inside the coefficients $a_{f}(n;\af)$, bringing us to our second result, formally given in Proposition \ref{prop:the_proposition}.
Write
$n_1 = \prod_{p\mid N}p^{n_p}$ for some $n_p\geq 0$ and $n_0\in\Nat$ with $(n_0,N)=1$. Then at the cusp $\af=a/q$ for some $q\mid N$ we have
\begin{multline}\label{eq:intro_cusps}
	a_f(n_1 n_0;\af)=  \Omega_{\chi,q,\delta(\af)} a_f(n_0,\infty) e\left( \frac{n\overline{a}}{\delta(\af)L} \right)\left(\frac{n_1}{\delta(\af)}\right)^{\frac{k}{2}}\\ \times\, \prod_{p\mid n_1} \omega_{\chi,p}\left(\frac{n_0n_1}{(n_1,p^{\infty})}\right)W_p\left(\left(\begin{matrix} 0&p^{n_p-d_{\pi_p}(q_p)}\\-1& -u_p(n_0)p^{q_p^{-1}} \end{matrix}\right)\right)
\end{multline}
where $W_p$ is the $p$-th component of the global Whittaker function associated to $f$ (see \eqref{eq:factorisation_of_W_varphi} below), its matrix argument is given in Proposition \ref{prop:the_proposition}; $\omega_{\chi,p}$ is the $p$-th component of the Hecke-character associated to $\chi$; and $\Omega_{\chi,q,\delta(\af)}\in\C_{\abs{z}=1}$ is defined in \eqref{eq:def_of_unitary_coeffi} and in particular satisfies ${\Omega_{\chi,q,\delta(\af)}}=1$ whenever $\chi = 1$ or $q=1$. By the results of \cite{As17}, one may evaluate the local factors at $p\mid n_1$ explicitly. Furthermore, combining \eqref{eq:intro_cusps} with \eqref{eq:intro_voronoi} we thus obtain a general Vorono\"i summation formula for $\GL_{2}$ in the $N$-$q$ aspect.

\begin{rem}
Formula \eqref{eq:intro_cusps} implies the well known phenomenon that the Fourier coefficients $a_f(n;\af)$ in general are not multiplicative objects.
Moreover, due to the possible large values of $W_{p}$, derived in \cite{As17}, the coefficients $a_f(n;\af)$ can be quite large in terms of the level $N$.
\end{rem}

We are able to give some novelty results as immediate corollaries to \eqref{eq:intro_cusps}. For example, in Corollary \ref{cor:ppprrp} we compute the bound
\begin{equation}
	\sum_{n\leq X} \vert a_f(n;\af)\vert^2
	\ll \frac{k^2N^{\epsilon}}{\delta(\af)^k}\left( X^{k+\epsilon} + X^{k-\frac{1}{2}+\epsilon}(L,\frac{N}{L})^k\right).\nonumber
\end{equation}
Another application of \eqref{eq:intro_cusps} is to derive an adelic approach to the classical Atkin--Lehner relations. We will see in \S \ref{sec:atkin_lehner} that, via \eqref{eq:intro_cusps}, the Atkin--Lehner involution arises naturally through local matrix identities. The corresponding identities relating Fourier coefficients at different cusps arise from several functional equations of the local Whittaker functions $W_p$.


\begin{rem}
 Such formulae as \eqref{eq:intro_voronoi} and \eqref{eq:intro_cusps} are naturally sensitive to the inducing information that determines the local representation attached to $f$ at $p$. Hence a Vorono\"i formula is inherently non-uniform in the $N$-$q$ aspect outside the extreme cases $(N,q)=1$ and $N\mid q$, where the local representations do not interact. We remark that our formula trivially agrees with the `go-to' reference  \cite[\S A]{KMV02} in such extreme cases (see Corollaries~\ref{cor:cl_vor_1} and~\ref{cor:cl_vor_2}).
Uniformity may nonetheless be recovered should one concede to looking at families of fixed (or restricted) representation type at $p$, in which case such a general Vorono\"i formula is quickly derived from the formulae given in \cite{As17}, in conjunction with \eqref{eq:intro_voronoi} and \eqref{eq:intro_cusps}.
\end{rem}

\begin{rem}
Our approach towards Vorono\"i summation does not directly rely on the functional equation of Dirichlet series. Instead it can be explained in terms of a strong Gelfand formation. (See \cite{Reznikov} for discussion on strong Gelfand formations.) The sides of the formation are formed by the inclusions $\{1\}\subset N \subset \GL_{2}$ and $\{1\} \subset wNw^{-1} \subset \GL_2$ where $N$ is the unipotent subgroup of $\GL_{2}$. Each inclusion comes with the strong multiplicity-one property. This enables one to expand the trivial period
\begin{equation}
	[\varphi \mapsto \varphi(g)] \in \Hom_{\{1\}}(\pi, \C), \nonumber
\end{equation} 
for an automorphic representation $\pi$ of $\GL_{2}$, in terms of well chosen model periods in the intermediate $\Hom$-spaces; one wall of the Gelfand formation of course gives rise to the Whittaker model. The equality arising from equating the expansions coming from each side determines the functional equation for a general Vorono\"i-type formula in the representation theoretic setting. Classically, this can be thought of as swapping from the Fourier expansion of a cusp form at $\infty$ to the Fourier expansion at $0$. This is because the group $N$ (resp.\ $wNw^{-1}$) corresponds to the stabiliser of $\infty$ (resp.\ $0$). We make this remark precise in Theorem~\ref{th:voronoi2}.
\end{rem}




\subsection*{Common notation}

Put $e(z):=e^{2\pi i z}$ for $z\in\C$. If $g=\left(\begin{smallmatrix}
a&b\\
c&d
\end{smallmatrix}\right)\in \GL_{2}(\C)$ we write $gz=\frac{az+b}{cz+d}$ and $j(g,z)=cz+d$. If moreover $g\in\SL_{2}(\Z)$ then, by convention, a Dirichlet character $\chi$ defines a function $\chi(g):=\chi(d)$.
Let $(a,b)$ and $[a,b]$ denote the greatest common divisor and the least common multiple of $a,b\in\Nat$, respectively. Additionally, if $c\in\Nat$ let $[a,b,c]:=[a,[b,c]]=[[a,b],c]$. We write $a\mid b^{\infty}$ to denote that there exists a $k\in\Nat$ such that $a\mid b^{k}$; similarly we let $(a,b^{\infty}):=\max_{k\geq1}(a,b^{k})$. We let $\lfloor y \rfloor$ and $\lceil y \rceil$ denote the floor and ceiling of $y\in \R$.
The $p$-adic valuation of $x\in\Qp$ is denoted by $v_p(x)$. If $A$ is a logical assertion, we use Dirac's symbol $\delta(A)$ to denote a $1$ if $A$ is true and a $0$ otherwise.


\section{Classical Fourier expansions at the cusps of \tp{$\Gamma_0(N)\bs\Hb$}{the upper half-plane modulo a congruence subgroup}}\label{sec:fourier_expn}

\subsection{Classical newforms}\label{sec:classical_newforms}

The results of this work apply specifically to newforms on $\Gamma_0(N)\bs\Hb$ which vanish in the cusps. These may come from either the world of Maa\ss\ forms of holomorphic modular forms.\footnote{We essentially consider the full spectrum of $\Gamma\bs\Hb$ from a representation theoretic point of view. The weight $k$ Maa\ss\ forms reduce to the weight $0$ and $1$ cases via the weight lowering operator (see \cite[Exercise~2.1.7]{bump}).
Our omission of the weight $1$ Maa\ss\ forms is but for notational convenience. Our results may be straightforwardly amended to describe such forms.}
We introduce notation here to deal with both of these cases. From a pedagogical perspective, much of the theory we introduce in this section and the next applies more generally, without the assumptions that the form is new or indeed cuspidal \cite{bump,gelbart}.
We shall uniformly refer to a function $f\colon \Hb\rightarrow\C$ as a \textit{cuspidal newform} of weight $k\in \Z_{\geq 0}$ if $f$ falls into one of the following two categories:
\begin{itemize}
\itemspacing
\item Cuspidal holomorphic Hecke newforms, in which case $k>0$.
\item Cuspidal Hecke--Maa\ss\ newforms of weight $k=0$. In this case fix $m\in\{0,1\}$ such that $f(-\overline{z})=(-1)^{m}f(z)$ and $\lambda_f=\frac{1}{4}+t_f^2$, the (Laplace) eigenvalue of $\Delta:=-y^2(\frac{\partial^2}{\partial x^2} + \frac{\partial^2}{\partial y^2})$ for which $\Delta f= \lambda_f f$.
\end{itemize}
We additionally assume the normalisation of \eqref{eq:f_normalisation}.
We consider such a cuspidal newform $f$ as fixed throughout our exposition.
For a Dirichlet character of conductor $M$ with $M\mid N$,
we say that a newform $f$ has level $N$ and nebentypus $\chi$ if for each $\gamma\in\Gamma_{0}(N)$ the weight $k$ modular identity holds:
\begin{equation}\label{eq:modular_identity}
	f\vert_{k}\gamma=\chi(\gamma)f
\end{equation}
where we apply the usual definitions such that for, $g=\left(\begin{smallmatrix}
a&b\\
c&d
\end{smallmatrix}\right)\in\GL_{2}(\R)$ with $\det g>0$, we have $(f\vert_{k}g)(z):=\deg(g)^{k/2}(cz+d)^{-k}f(gz)$ where $gz:=\frac{az+b}{cz+d}$, and when moreover $g\in\SL_{2}(\Z)$ then $\chi(g):=\chi(d)$. In particular, for $f\neq 0$ we further impose that $\chi(-1)=(-1)^{k}$.


We are able to speak of both genres of newform in the same breath by introducing the notation
\begin{equation}\label{eq:kappa_def}
	\kappa_{f}(y):=
	\begin{cases}
		\frac{\sgn(y) + 1}{2}  e^{-2\pi y}&\text{if } f \text{ is a holomorphic modular form}\\
		\sgn(y)^{m}\abs{y}^{1/2}K_{it_f}(2\pi\abs{y})&\text{if } f \text{ is a Maa\ss\ form.}\\
	\end{cases}
\end{equation}
Since $\left(\begin{smallmatrix}
1&1\\
0&1
\end{smallmatrix}\right)\in \Gamma_{0}(N)$ for any $N\geq 1$, we have in particular that $f(z+1)=f(z)$ by \eqref{eq:modular_identity}. Then with the specialist notation in \eqref{eq:kappa_def}, we deduce the usual Fourier expansion for $f$ (at the cusp $\infty$)
\begin{equation}\label{eq:fourier_kappa}
	f(x+iy)=\sum_{n\in\Z} a_{f}(n) \kappa_{f}(ny)e(nx),
\end{equation}
the cuspidal condition implying $a_f(0)=0$. Without loss in generality we assume the normalisation
\begin{equation}\label{eq:f_normalisation}
a_f(1)=1.
\end{equation}
Then the Hecke eigenvalues of $f$, denoted by $\lambda_f(n)$, satisfy $a_f(n)=\lambda_f(n)n^{\frac{k-1}{2}}$.



\subsection{The cusps of \tp{$\Gamma_{0}(N)\bs\Hb$}{Y0(N)}}\label{sec:cusps}

A detailed account of the construction of the cusps of the Riemann surface $\Gamma_{0}(N)\bs\Hb$ is given in \cite[\S 3.4.1]{nps}, to which we refer. We denote this set of cusps by $\Cc(N)$; that is, the set of boundary points of $\Hb$ (modulo the left M\"obius action of $\Gamma_{0}(N)$) that are stabilised by a non-scalar element of $\Gamma_{0}(N)$, or simply $\Cc(N)=\Gamma_{0}(N)\bs\Pb^{1}(\Q)$. Indeed, $\SL_{2}(\Z)$ acts transitively on $\Pb^{1}(\Q)$ and the stabiliser of the point $\infty=[1:0]\in \Pb^{1}(\Q)$ is $\Gamma_{\infty}=\{\pm \left(\begin{smallmatrix}
1&n\\
0&1
\end{smallmatrix}\right) : n\in\Z\}$. Thus $\Cc(N)$ is identified with $\Gamma_{0}(N)\bs \SL_{2}(\Z)/\Gamma_{\infty}$ as a left $\Gamma_{0}(N)$-set. We write its elements as $\af=\sigma^{-1}\infty$ for a choice of representatives $\{\sigma\}$ of $\Cc(N)$.

On the other hand, there is a (transitive) right action of $\SL_{2}(\Z)$ on $\Pb^{1}(\Z/ N\Z)$ which, as in \cite[\S 3.4.1]{nps}, may be used to parameterise the set of cusps by classes $[q:d]\in\Pb^{1}(\Z/ N\Z)$ such that $q\mid N$ and $d\in (\Z/(q, N/q)\Z)^{\times}$. From this perspective one refers to $[q:d]$ as the ``fraction'' ${d}/{q}$ and calls $q$ the \textit{denominator} of the cusp. It follows that the number of cusps of denominator $q$ is $\phi((q,N/q))$ and furthermore $\#\Cc(N)=\sum_{q\mid N}\phi((q,N/q))$. In this description, the cusp $\af=\infty$ corresponds to $[0:1]$ and is the unique cusp of denominator $N$.

To summarise these constructions, any cusp $\af\in\Cc(N)$ may be identified by an element
\begin{equation}\label{eq:sigma_inverse}
	\sigma^{-1}=
	\begin{pmatrix}
		a&b\\q&d
	\end{pmatrix}\in\SL_{2}(\Z)
\end{equation}
such that $\sigma\af=\infty$. Then the cusp $\af$ has denominator $q$; it corresponds to the class $[q:d]\in\Pb^{1}(\Z/ N\Z)$ (as above); and if $\sigma\not\equiv 1\in \Gamma_{0}(N)\bs \SL_{2}(\Z)/\Gamma_{\infty}$ then $\af$ is equal to the rational cusp $a/q\in\Pb^{1}(\Q)$ where $(a,N)=1$ and $ad\equiv 1 \Mod{(q,N/q)}$.

The \textit{width} of a cusp $\af\in\Cc(N)$ is defined to be $w(\af)=[\Gamma_{\infty}:\Gamma_{\infty}\cap\sigma\Gamma_{0}(N)\sigma^{-1}]$ where $\sigma\in \Gamma_{0}(N)\bs \SL_{2}(\Z)/\Gamma_{\infty}$ satisfies $\sigma\af=\infty$. Equivalently, a more tactile definition may be given by noting that $w(\af)$ is the least integer $n\geq 1$ such that the stabiliser of $\af\in \Pb^{1}(\Q)$, under the left action of $\SL_{2}(\Z)$, contains $\left(\begin{smallmatrix}1&n\\0&1\end{smallmatrix}\right)$. In fact, one may compute the width precisely as $w(\af)=N/(q^{2},N)$, or equivalently $w(\af)=[q^{2},N]/q^{2}$ (see \cite[\S 3.4.1]{nps}).

Later on it shall be handy to take any fraction $a/q$ with $(a,q)=1$ and consider the equivalence class defining the cusp $\af$. This relates to our notion above as follows:
with $\sigma^{-1}$ as in \eqref{eq:sigma_inverse}, let $r,s\in\Z$ such that $qr+Nds = (q,N)$ and rescale $\sigma$ so that
\begin{equation}
	\left(\begin{matrix} ar  & * \\ (q,N) & * \end{matrix}\right) = 
	\left(\begin{matrix} 1 & -\frac{N}{(q,N)} bs \\0&1 \end{matrix}\right)	\sigma^{-1}	\left(\begin{matrix}r &* \\ Ns&* \end{matrix}\right).\nonumber
\end{equation}
Replacing the coefficient $ar$ by its unique minimal representative $a'\geq 0$ modulo $(q,N)$, another representative of the cusp $\af$ is then $a'/(q,N)$, the upshot being that the latter has denominator dividing $N$. 

\subsection{The classical Fourier expansion at an arbitrary cusp} \label{sec:fourier_classical}

For a cusp $\af\in\Cc(N)$, pick a representative $\sigma\in\SL_{2}(\Z)$ such that $\sigma\af=\infty$.
The expansion of $f$ about $\af$ is given by the expansion of $f\vert_{k}\sigma^{-1}$ at the cusp $\infty$. If the nebentypus were trivial, $\chi=1$, then $f\vert_{k}\sigma^{-1}$ would be periodic on vertical strips of width $w(\af)$ in $\Hb$. However, to account for general characters of conductor $M\mid N$, we introduce the \textit{extended width}
\begin{equation}\label{eq:extended_width}
	\delta(\af):=w(\af)\frac{M}{(qw(\af),M)}=\frac{[Mq,q^{2},N]}{q^{2}}.
\end{equation}
As with $w(\af)$, the extended width only depends on the equivalence class of $\af$. Suppose that $\sigma=\left(
	\begin{smallmatrix}
		d&-b\\
		-q&a
	\end{smallmatrix}\right)$;
that is, $\af=\sigma^{-1}\infty$ corresponds to the cusp $a/q$ as described in \S \ref{sec:cusps}.
Then we have
\begin{equation*}
	\sigma^{-1}\begin{pmatrix}
		1&w(\af)t\\
		0&1
	\end{pmatrix}\sigma=
	\begin{pmatrix}
		1- aqw(\af)t& a^{2}w(\af)t\\
		-q^{2}w(\af)t& 1+aqw(\af)t
	\end{pmatrix}
\end{equation*}
for each $t\in\Z$. A few applications of the modular identity \eqref{eq:modular_identity} imply that
\begin{equation*}
	(f\vert_{k}\sigma^{-1})(z + w(\af)t)=\chi(1- aqw(\af)t)(f\vert_{k}\sigma^{-1})(z).
\end{equation*}
As the conductor of $\chi$ is $M$ we see that $\chi(1- aqw(\af)t)=1$ if and only if $M\mid qw(\af)t$, or equivalently, $\frac{M}{(qw(\af),M)}\mid t.$
Hence, our definition of $\delta(\af)$ is justified in that it is the minimally chosen positive integer satisfying the periodicity relation
$$(f\vert_{k}\sigma^{-1})(z+\delta(\af))=(f\vert_{k}\sigma^{-1})(z).$$
The resulting Fourier expansion $f$ at the cusp $\af$ is of the form
\begin{equation}\label{eq:fourier_expn_cusps}
	(f\vert_{k}\sigma^{-1})(x+iy)=\sum_{n\in\Z} a_{f}(n;\af)\kappa_{f}(ny/\delta(\af)) e(nx/\delta(\af)).
\end{equation}
(One may refer to \cite[Theorems 3.4.1 and 3.7.4]{goldfeld-hundley-1} for both the holomorphic and Maa\ss\ cases, respectively.)

When $\af=\infty$ we recover the familiar coefficients $a_{f}(n;\af)=a_{f}(n)$; however, unlike $a_{f}(n)$, the coefficients $a_{f}(n;\af)$ are not multiplicative in general (the reason for this can be found in Section~4.2 below).
We refer to $f$ as a \textit{cusp form} whenever its constant terms vanish, $a_{f}(0;\af)=0$, with respect to each $\af\in\Cc(N)$. By definition, if $f$ is holomorphic then its Laurent expansion contains no negatively indexed terms: $a_{f}(n;\af)=0$ for all $n<0$.

\begin{rem}
Note that $\delta(\af)=w(\af)$ whenever the conductor $M$ is sufficiently small; for example, if $M\mid N_{1}$ where $N_{1}$ is the smallest integer such that $N\mid N_{1}^{2}$. On the other hand, if $q$ is sufficiently square-full above its prime divisors then $\delta(\af)=w(\af)$, likewise.
As with $w(\af)$, the extended width depends only on $(q,N)$ as can be seen from $\delta(\af)=w(\af)\frac{M}{(qw(\af),M)}$ and the fact that $(M\mid N)$.
\end{rem}

Finally, suppose we wish to choose a different representative $\tau$ satisfying $\tau\af=\infty$. This includes representing $\af$ by a different fraction without necessarily asking for the numerator and denominator to be coprime. Then $\tau^{-1}= \smat{1}{m}{}{1} {\sigma}^{-1}\gamma$ for some $m\in \Z$ and $\gamma\in \Gamma_0(N)$. Applying \eqref{eq:modular_identity} we obtain
\begin{equation}\label{eq:classical_change_of_scaling}
\begin{array}{rc>{\ds}l}\vspace{0.1in}
f\vert_k\tau^{-1}(x+iy)
		&=&\chi(\gamma) f\vert_k \sigma^{-1} (x+m+iy) \\		
		&=&\chi(\gamma)\sum_{n\in\Z} a_{f}(n;\af)\kappa_{f}(ny/\delta(\af)) e(n(x+m)/\delta(\af)).
\end{array}
\end{equation}
This shows explicitly that replacing the matrix $\sigma$ by an equivalent one $\tau$ only skews the Fourier coefficients by the root of unity $\chi(\gamma) e(nm/\delta(\af))$.

\section{Whittaker-Fourier expansions and adele groups}\label{sec:whittaker}

In this section we take an expository route to give explicit realisations of classical forms recast in the adelic theory of representations for $G:=\GL_{2}$. We hope to exhibit the advantage gained by strong approximation when working with adelic constructions -- this is the key insight in converting global ramification problems into $p$-adic analytic ones.


\subsection{Background on automorphic forms on \tp{$\GL_{2}(\A)$}{GL(2,A)}}

\subsubsection{Notation for adele groups}

Let $\A$ denote the \textit{adele ring of $\Q$} and $\A_{\f}$ its ring of finite adeles; these are given respectively by the restricted direct products $\A=\prod'_{p\leq\infty}\Qp$ and $\A_{\f}=\prod'_{p<\infty}\Qp$ with respect to the additive subgroups $\Zp\leq \Qp$.
We reserve the letter $G$ to denote $\GL_{2}$ and introduce shorthand for the $p$-adic groups $G_{p}=\GL_{2}(\Qp)$, $G_{\infty} =\GL_{2}(R)$ and their maximal compact subgroups $ K_{p}=\GL_{2}(\Zp)$, $K_{\infty}=\Orth_{2}(\R)$, respectively. At finite places $p<\infty$ we shall consider the subgroups $K_{p}(N)\leq K_{p}$ containing those matrices $\left(\begin{smallmatrix}
a&b\\
c&d
\end{smallmatrix}\right)\in K_{p}$ such that $c\in N\Zp$. Further, the finer subgroups $K_{1,p}(N)\subset K_p(N)$ (resp.\ $K_{1,p}'\subset K_p(N)$) are defined to contain the elements $K_{p}(N)$ such that $d \in 1+N\Zp$ (resp.\ $a\in 1+N\Zp$).
The adele group of $\GL_{2}$ is then the restricted direct product $G(\A)=\prod'_{p\leq\infty}G_{p}$ with respect to the subgroups $K_{p}$. We consider the global congruence subgroups of $G(\A)$ defined by $K_{0}(N)=\prod_{p<\infty} K_{p}(N)$, $K_{1}(N)=\prod_{p<\infty} K_{1,p}(N)$ and $K_{1}'(N)=\prod_{p<\infty} K_{1,p}'(N)$.

\subsubsection{Notation for matrix groups} For a commutative unital ring $R$ and a subgroup $H\leq G(\R)$ write $H^{+}$ for the subgroup of matrices $h\in H$ such that $\det(h)>0$. We consider certain elements given by
\begin{equation*}
z(\lambda):=\begin{pmatrix}
\lambda&\\&\lambda
\end{pmatrix},\quad
n(x):=\begin{pmatrix}
1&x\\&1
\end{pmatrix},\quad
a(y):=\begin{pmatrix}
y&\\&1
\end{pmatrix},\quad
w:=\begin{pmatrix}
&1\\-1&
\end{pmatrix}
\end{equation*}
for each $\lambda,y\in R^{\times}$ and $x\in R$. These elements determine the familiar subgroups $Z(R):=\{z(\lambda):\lambda\in R^{\times}\}$, the centre of $G(R)$; $N(R):=\{n(x):x\in R\}$, the unipotent matrices; $A(R):=\{a(y):y\in R^{\times}\}$; and the upper triangular Borel subgroup $B:=ZNA$. 

\subsubsection{Strong approximation}

We diagonally embed $\Q\emb\A$ as a discrete subgroup. The quotient $\Q\bs\A$ is compact and we have the strong approximation theorem
\begin{equation}\label{eq:strong_approx_A}
\A\isom \Q+\big([0,1)\times {\textstyle\prod_{p<\infty}}\Zp\big)
\end{equation}
(see \cite[Corollary 5-9]{ramakrishnan}). We additionally have strong approximation theorems for the adele groups $\Ax=\GL_{1}(\A)$ and $G(\A)=\GL_{2}(\A)$ which state the following:
\begin{eqnarray}
\label{eq:strong_approx_Ax}
\textstyle\Ax	& \isom &	\Q^{\times}\cdot\big(\R_{>0}\times{\textstyle \prod_{p<\infty}}\Zpx\big)\\
\label{eq:strong_approx_G}
G(\A)				& \isom &	G(\Q)\cdot\big( G_{\infty}^{+}\times K_{0}(N)\big).
\end{eqnarray}
for each $N\geq 1$ (see \cite[Ch.~3]{gelbart}). The adelic congruence subgroups are then related to the classical ones via
\begin{equation}\label{eq:adele_gamma}
G(\Q)\cap \big( G_{\infty}^{+}\times K_{0}(N)\big) = \Gamma_{0}(N).
\end{equation}

In the spirit of clarity, we specify the place-value of matrices $\gamma\in G(\Q)$ by defining the inclusions $\iota_{p}\colon G(\Q)\emb G(\Qp) \emb G(\A)$ for each $p\leq\infty$ and the diagonal imbedding $\iota_{\f}\colon G(\Q)\emb G(\Af) \emb G(\A)$.

\subsubsection{Dirichlet characters {\`a} la Hecke}

A Dirichlet character $\chi$ (of conductor $M$) may be associated to a Hecke character $\omega_{\chi}\colon \Ax\rightarrow\Cx$ via the diagram
\begin{equation*}
	\omega_{\chi}\colon \Ax
	\arrup{\longrightarrow}{pr}\prod_{p\mid M}\Zpx
	\arrup{\longrightarrow}{\Ch_{M}}(\Z/M\Z)^{\times}
	\arrup{\longrightarrow}{\chi^{-1}}\Cx,
\end{equation*}
where $\pr$ denotes the projection of $\R_{>0}\times\prod_{p<\infty}\Zpx$ after applying strong approximation \eqref{eq:strong_approx_Ax} to $\Ax$; and $\Ch_{M}$ denotes the cannonical surjection induced by the Chinese remainder theorem, $\prod_{p}\Zpx/(1+p^{m_{p}}\Zp)\isom (\Z/M\Z)^{\times}$ for $M=\prod_{p}p^{m_{p}}$.
Explicitly, we highlight the inverse; $\omega_{\chi}(x) = \chi(\Ch_{M}(\pr(x)))^{-1}$.
As with all Hecke characters, we have the tensor factorisation $\omega_{\chi} = \otimes_{p \leq \infty} \omega_{\chi,p}$.
Observe in particular that $\omega_{\chi,\infty}\vert_{\R_{+}} = 1$.
See \cite[\S 12.1]{knightly-li} for extended details.
Abusing notation, we extend the character $\omega_{\chi}=\otimes_{p \leq \infty} \omega_{\chi,p}$ to a character of $K_0(N)$ by setting
\begin{equation}\label{eq:char_K0N_extension}
	\omega_{\chi,p}(k) = \begin{cases}
		\omega_{\chi,p}(d) &\text{ for } k= \left(\begin{matrix} a & b \\ p^{v_p(N)}c & d \end{matrix}\right)\in K_p(N) \text{ and }p\mid N, \\
		1 &\text{ for }p\nmid N.
	\end{cases} \nonumber
\end{equation}

\subsubsection{Adelisation of classical modular forms}

We refer the unfamiliar reader to \cite{gelbart} for a friendly introduction and to the following construction; see also \cite[\S 12]{knightly-li}.
For our classical newform $f$, the `adelisation process' is to associate an \textit{automorphic form}
$\varphi\colon G(\Q)\bs G(\A)\rightarrow \C$ to $f$ by defining
\begin{equation}\label{eq:adelisation_definition}
\varphi(g):=\omega_{\chi}(k_{0})(f\vert_{k}g_{\infty})(i)
\end{equation}
where $g=\gamma g_{\infty} k_{0}\in G(\A)$ is decomposed as in \eqref{eq:strong_approx_G} with $\gamma\in G(\Q)$, $g_{\infty}\in G_{\infty}^{+}$ and $k_{0}\in K_{0}(N)$.
By choice of $f$, we consider $\varphi$ as fixed throughout our exposition.
The central character of $\varphi$ is given by $\omega_{\chi}$, which is understood as follows.
Via \eqref{eq:strong_approx_Ax}, choose $r=r_{\Q}r_{\infty}r_{\f}\in\Ax$ with $r_{\Q}\in\Qx$, $r_{\infty}\in\R_{>0}$ and $r_{\f}\in\prod_{p<\infty}\Zpx$.
Then one computes
$\varphi(z(r)g) = \omega_{\chi}(z(r_{\f})k_0)(f\vert_k z(r_{\infty})g_{\infty})(i) = \omega_{\chi}(z(r_{\f})) \varphi(g).$
But indeed $\omega_{\chi}(z(r_{\f})) = \omega_{\chi}(r)$ as $\omega_{\chi}(r_{\infty})=1$ by construction.
In the adelic setting, the central character also finds itself as the one dimensional representation of $K_0(N)$ generated by $\varphi$ since
\begin{equation}
		\varphi(gk_0) = \omega_{\chi}(k_0)\varphi(g)\label{eq:right_translation_by_K0}
\end{equation} 
for all $k_0\in K_0(N)$ and $g\in G(\A)$. In particular, $\varphi$ is right $K_1(N)$-invariant.

\subsubsection{Automorphic representations}


The right regular action\footnote{There is a caveat at $p=\infty$; we instead have a (commuting) action of $(\Lie(G_{\infty}),K_{\infty})$.} of $G(\A)$ is realised in the space $L^{2}(Z(\A)G(\Q)\bs G(\A),\omega_{\chi})$, containing only forms with central character $\omega_{\chi}$. Let $\pi$ denote the fixed automorphic representation representation of $G(\A)$ generated by $\varphi$ (or, by abuse of language, $f$). The property that $f$ is a cusp form is equivalent to $\pi$ being contained in the cuspidal part of the spectrum of $L^{2}(Z(\A)G(\Q)\bs G(\A),\omega_{\chi})$; that $f$ is moreover a newform directly implies that $\pi$ is irreducible.
Let $\tilde{\pi}$ denote the contragredient representation attached to $\pi$. Then $\tilde{\pi}$ has central character $\omega_{\chi}^{-1}$ and satisfies the isomorphism $\tilde{\pi}\isom(\omega_{\chi}^{-1}\circ\det)\otimes\pi$. 
The tensor product theorem (see \cite{bump}) states that an automorphic representation $\pi$ may be realised as a restricted tensor product $\pi\isom\otimes'_{p}\pi_{p}$ where each $\pi_{p}$ is some (complex) representation of $G_{p}$. This algebraic construction is tied fundamentally to the theory of Euler products of $L$-functions in that the coefficients $\lambda_{\pi}(n)$ of 
\begin{equation}
L(s,\pi) = \prod_{p\leq\infty}L(s,\pi_p)=L(s,\pi_\infty)\sum_{n\geq 1}\frac{\lambda_{\pi}(n)}{n^{s}}
\end{equation}
are local objects. It is the $G(\Q)$-invariance of $\pi$ that implies the global functional equation about $s\mapsto 1-s$ involving the epsilon factor $\varepsilon(s,\pi)=\prod_{p\leq\infty}\varepsilon(s,\pi_p)$. Note that in our particular example that $\lambda_\pi=\lambda_f$ and the completed $L$-function attached to $f$ is given by $\Lambda(s,f)=L(s,\pi)$.

\subsection{The Whittaker-Fourier Expansion}

Let $\psi \colon \Q\bs\A\rightarrow\Cx$ be the standard additive character on $\A$ defined as follows: $\psi=\otimes_{p}\psi_{p}$ where $\psi_{\infty}(x_{\infty})= e(x_{\infty})$ and $\psi_{p}(x_{p})=1$ if and only if $x_{p}\in\Zp$.
A smooth function $W\colon G(\A)\rightarrow\C$ is called a $\psi$-Whittaker function if $W(n(x)g)=\psi(x)W(g)$ is satisfied for each $x\in\A$ and $g\in G(\A)$. By the uniqueness of Whittaker models for $G(\A)$, there exists a subspace of such functions which, under the right regular action of $G(\A)$, is isomorphic to $\pi$; this subspace itself is called the \textit{Whittaker model} and is denoted by $\Wc(\pi,\psi)$. The implied $G(\A)$-isomorphism may be explicated via the map
\begin{equation}\label{eq:whittaker_isomorphism}
\varphi(g)\longmapsto W_{\varphi}(g)=\int_{\Q\bs\A}\varphi (n(x)g) \,\overline{\psi(x)}\,dx
\end{equation}
where we take the invariant measure $dx$ on $\A$, which naturally descents to a probability measure on $\Q\backslash\A$. It is then apparent that $\Wc(\pi,\psi)$ contains the Fourier transforms of the function $x\mapsto \varphi(n(x)g)$ on $\Q\bs\A$ whose dual group is identified as $\Q$ by considering the characters $x\mapsto\psi(\xi x)$ for each $\xi\in\Q$. By Fourier inversion, we thus derive an expansion of $\varphi$ in terms of $W_{\varphi}$ so that
\begin{equation}\label{eq:whittaker_expn}
\varphi(g)=\sum_{\xi\in\Qx}W_{\varphi}(a(\xi)g)
\end{equation}
where we have implicitly used the cuspidal property of $\varphi$ to deduce the vanishing of the constant term (at $\xi=0$).
This adelic Fourier expansion is rather marvellous in that it encodes the classical Fourier expansion \eqref{eq:fourier_expn_cusps} of $f$ at all cusps $\af$ of $\Gamma_0(N)\bs \Hb$ simultaneously.
We now conclude this section with some remarks on using the vantage point of \eqref{eq:whittaker_expn} to compute the coefficients $a_{f}(n;\af)$.

\subsection{Adelic realisation of classical Fourier coefficients}

Let $z=x+iy\in\Hb$ and define $g_{z}=n(x)a(y)\in G_{\infty}^{+}$. Evaluating $\varphi$ on this matrix we recover
\begin{equation}\label{eq:varphi_at_gz}
\varphi(g_{z})=(f\vert_{k}g_{z})(i)=y^{k/2}f(z).
\end{equation}
Shifting to an arbitrary cusp $\af=\sigma^{-1}\infty$, we combine \eqref{eq:whittaker_expn} with our observation above in \eqref{eq:varphi_at_gz} to find
\begin{equation}\label{eq:f_to_whitt_expn}
\begin{array}{rc>{\ds}l}\vspace{0.1in}
(f\vert_{k}\sigma^{-1})(z)&=&y^{-k/2}(f\vert_{k}\sigma^{-1}g_{z})(i)\\\vspace{0.1in}
&=&y^{-k/2}\varphi(g_{z}\iota_{\f}(\sigma))\\
&=&y^{-k/2}\sum_{\xi\in\Qx}W_{\varphi}(a(\xi)g_{z}\iota_{\f}(\sigma)).
\end{array}
\end{equation}
Note the conceptual part played by strong approximation \ref{eq:strong_approx_G} for this identity to hold. The corollary to \eqref{eq:f_to_whitt_expn} is that we may determine the values $W_{\varphi}(a(\xi)g_{z}\iota_{\f}(\sigma))$ in two ways: firstly by returning full circle in recovering the classical Fourier expansion \eqref{eq:fourier_expn_cusps}, and secondly by explicating their values locally via the associated representation theory. Beginning with the former, a special case of the following result is given (for amusement) in \cite[Lem.~3.6]{gelbart}. We give full generality here.

\begin{prop}\label{prop:whittaker_coeffs}
Let $g_{z}=n(x)a(y)$ for $z=x+iy\in\Hb$ and let $\af\in\Cc(N)$ with $\sigma\af=\infty$. Then for $\xi\in\Qx$ we have that
\begin{equation*}
W_{\varphi}(a(\xi) g_{z}\iota_{\f}(\sigma))=
y^{k/2}a_{f}(n;\af)\kappa_{f}(ny/\delta(\af))e(nx/\delta(\af))
\end{equation*}
if $\xi=n/\delta(\af)$ for some $n \in \Z$ and $W_{\varphi}(a(\xi) g_{z}\iota_{\f}(\sigma))=0$ otherwise.
\end{prop}

\begin{proof}

Using the left invariance of $\varphi$ by $G(\Q)$ and of the additive Haar measure on $\A$, $d(\xi x')=\abs{\xi}_{\A}dx'=dx'$ for each $\xi\in\Qx$, we compute
\begin{equation*}
\begin{array}{rc>{\ds}l}\vspace{0.1in}
W_{\varphi}(a(\xi) g_{z}\iota_{\f}(\sigma))&=&\int_{\Q\bs\A}\varphi(n(x')a(\xi)g_{z}\iota_{\f}(\sigma))\psi(-x')\,dx'\\\vspace{0.1in}
&=&\psi_{\infty}(\iota_{\infty}(\xi)x)\int_{\Q\bs\A}\varphi(n(x')a(y)\iota_{\f}(\sigma))\psi(-\xi x')\,dx'\\
&=&e(\xi x)C_{\varphi}(\xi; y,\sigma)
\end{array}
\end{equation*}
where we have defined
\begin{equation*}
C_{\varphi}(\xi; y,\sigma):=\int_{\Q\bs\A}\varphi(n(x')a(y)\iota_{\f}(\sigma))\psi(-\xi x')\,dx'.
\end{equation*}

Firstly, let $\xi=n/\delta(\af)$ with $n\in\Z$. Consider the fundamental domain
\begin{equation}\label{eq:strong_approx_B}
\Q\bs\A=\R/\delta(\af)\Z\times \prod_{p<\infty}\delta(\af)\Zp
\end{equation}
where we associate $\R/\delta(\af)\Z=[0,\delta(\af))$. This is chosen using strong approximation \eqref{eq:strong_approx_A}.
For any $x_{p}\in\delta(\af)\Zp$ we have
$\sigma^{-1}n(x_{p})\sigma\in K_{p}(N)$, considering the inclusion $\Gamma_{0}(N)\emb K_{p}(N)$. Then \eqref{eq:right_translation_by_K0} implies that
\begin{equation*}
\varphi(n(x')a(y)\iota_{\f}(\sigma))=\varphi(\iota_{\infty}(\sigma^{-1})n(x_{\infty})a(y))\omega_{\chi}(\sigma^{-1}n(x_{\f})\sigma)\\
\end{equation*}
writing $x'=x_{\infty}x_{\f}$ with $x_{\f}=(x_{p})$ as in \eqref{eq:strong_approx_B}. By construction (of $\delta(\af)$), we have
\begin{equation}\label{eq:character_is_one}
	\omega_{\chi}(\sigma^{-1}n(x_{\f}\delta(\af))\sigma)=\prod_{p \mid M}\omega_{\chi,p}(1+aqx_{p})=1
\end{equation}
for $\sigma^{-1}=\smat{a}{b}{q}{d}$. Evaluating $\varphi(\iota_{\infty}(\sigma^{-1})n(x_{\infty})a(y))$ by definition as in \eqref{eq:adelisation_definition} we find
\begin{equation*}
\begin{array}{rc>{\ds}l}\vspace{0.1in}
C_{\varphi}(n/\delta(\af); y,\sigma)&=&\frac{1}{\delta(\af)}y^{k/2}\int_{0}^{\delta(\af)}(f\vert_{k}\sigma^{-1})(x_{\infty}+iy)e(-nx_{\infty}/\delta(\af))\,dx_{\infty}\\
&=&y^{k/2}a_{f}(n;\af)\kappa_{f}(ny/\delta(\af))
\end{array}
\end{equation*}
as desired. This last step follows by applying orthogonality of additive characters after inserting \eqref{eq:fourier_expn_cusps}. 

Now suppose that $\delta(\af)\xi\not\in \Z$. Then there exists a prime $p<\infty$ and an integer $m\geq 1$ such that $\xi=up^{-m}\delta(\af)^{-1}$ with $\abs{u}_{p}=1$. Consider the matrix $\iota_{p}(n(\delta(\af)p^{m-1}))$. As $w(\af)\mid \delta(\af)$, one has $n(\delta(\af)p^{m-1})\in\sigma\Gamma_{0}(N)\sigma^{-1}$, as in \S \ref{sec:cusps}, and hence $\sigma^{-1}n(\delta(\af)p^{m-1})\sigma\in K_{p}(N)$. By \eqref{eq:right_translation_by_K0} this element acts by the scalar $\omega_{\chi}(\sigma^{-1}n(\delta(\af)p^{m-1})\sigma)$, whence we deduce
\begin{equation*}
C_{\varphi}(\xi; y,\sigma)= \omega_{\chi}(\sigma^{-1}n(\delta(\af)p^{m-1})\sigma)^{-1}\psi_{p}(up^{-1}) C_{\varphi}(\xi; y,\sigma).
\end{equation*}
But $\omega_{\chi}(\sigma^{-1}n(\delta(\af)p^{m-1})\sigma)^{-1}=1$ by \eqref{eq:character_is_one}, hence $C_{\varphi}(\xi; y,\sigma)=0$ as $up^{-1}\not\in\Zp$.
\end{proof}

\section{Computing Fourier coefficients via Whittaker functionals}\label{sec:computing}

The key application of Proposition \ref{prop:whittaker_coeffs} is that we can now give an explicit product formula for the classical coefficients $a_{f}(n;\af)$ in terms of the factorisation of $W_{\varphi}$ into local Whittaker functionals; this is stated in Proposition \ref{prop:the_proposition}. We then exhibit several immediate corollaries to this product formula.

\subsection{Local Whittaker new vectors}


At each $p\leq\infty$ one may construct local Whittaker models $\Wc(\pi_{p},\psi_{p})\isom\pi_{p}$, analogous to the global case. In fact, the global Whittaker model factorises as $\Wc(\pi,\psi)=\otimes'\Wc(\pi_{p},\psi_{p})$ for $\pi\isom\otimes_{p}'\pi_{p}$. In particular the function $W_{\varphi}$, defined by \eqref{eq:whittaker_isomorphism}, factorises directly as a product of local ones
\begin{equation}\label{eq:factorisation_of_W_varphi}
W_{\varphi}=\otimes_{p\leq \infty} W_{p}
\end{equation}
which we normalise here by choosing $W_p(1)=1$ for all $p< \infty$. Using Proposition \ref{prop:whittaker_coeffs} this pins down $W_{\infty}$ on $A(\R)$ to be
\begin{equation}\label{eq:W_infty}
	W_{\infty}(a(y))= W_{\infty}(a(y))\prod_{p<\infty}W_{p}(1)=W_{\varphi}(a(y))= y^{k/2}\kappa_f(y). 
\end{equation}
and thus on all of $G(\R)$ by the Iwasawa decomposition. In particular, we find $W_{\varphi}(1) = \kappa_f(1)$.
At the finite primes $p\nmid N$,  right invariance by $K_p$ determines the normalised spherical vector $W_p$ uniquely: by the Iwasawa decomposition it is enough to determine $W_p$ on $A(\Qp)$ which is described by the well-known formula
\begin{equation}\label{eq:shintani}
	W_{p}(a(n)) = \abs{n}_p^{\frac{1}{2}}\lambda_{f}(p^{v_p(n)})
\end{equation}
for any $n\in \Nat$ (see \cite[Th.~4.6.5]{bump}). In fact \eqref{eq:shintani} this holds more generally for \textit{all} $p<\infty$ (see \cite[Section~2.4]{ralf-newforms}).

The theme of the present article is to understand the new vectors $W_p$ when $p\mid N$. By a generalisation of Atkin--Lehner theory \cite{casselman} the new vector is uniquely determined by the property that $W_p\vert_{K_{1,p}(v_p(N))} =1$. However, the interesting twist in such ramification problems is that a complete description of $W_p$ is not obtained by its values \eqref{eq:shintani} on $A(\Q_p)$, alone.

By Proposition \ref{prop:whittaker_coeffs}, \eqref{eq:W_infty} and \eqref{eq:shintani} we recover the well known identity
\begin{equation}\label{eq:coeffs_infty_muliplicative}
	a_f(n) = \frac{W_{\infty}(a(n))}{\kappa_f(n)}\prod_{p<\infty} W_{p}(a(n)) = n^{\frac{k}{2}} \prod_{p<\infty}\abs{n}_p^{\frac{1}{2}} \lambda_f(a(p^{v_p(n)}))   = n^{\frac{k-1}{2}}\lambda_{f}(n)
\end{equation}
for $n>0$. Implicit in \eqref{eq:coeffs_infty_muliplicative} is the well-known fact that the Fourier coefficients of $f$ at the cusp $\infty$ are multiplicative.
The main result of this section, Proposition \ref{prop:the_proposition}, depicts explicitly the degeneracy in the multiplicative property in $a_f(n,\af)$ at other cusps $\af\neq \infty$.

\subsection{A product formula for Fourier coefficients}

Now let us derive the product formula for the coefficients $a_f(n;\af)$, as a generalisation to \eqref{eq:coeffs_infty_muliplicative}. To highlight the ramified behaviour, let us factorise an arbitrary non-zero integer
$$n= \pm n_{0}\prod_{p\mid N}p^{n_{p}}$$
where $n_0\geq 1$ with $(n_0,N)=1$.
We extend Proposition \ref{prop:whittaker_coeffs} here, alongside \eqref{eq:W_infty} and \eqref{eq:shintani}, to allow for $n<0$. If $f$ is holomorphic, we intend the archimedean component $\kappa_f$ to account for the vanishing on $n<0$. We obtain
\begin{equation}\begin{array}{rc>{\ds}l}\vspace{0.1in}\label{eq:factorisation_we_need}
a_{f}(n;\af)&=&\frac{ W_{\infty}(a(n/\delta(\af)))}{\kappa_f(n/\delta(\af))}\prod_{p<\infty}W_{p}(a(n/\delta(\af))\sigma)\\
&=&\bigg(\dfrac{\abs{n}}{\delta(\af)}\bigg)^{k/2}\prod_{p\mid n_{0}}W_{p}(a(n_{0}))\ \prod_{p\mid N}W_{p}(a(n/\delta(\af))\sigma).
\end{array}
\end{equation}
We evaluate the $p\mid n_0$ sum via \eqref{eq:shintani} as before so that
\begin{equation*}
a_{f}(n;\af)=\dfrac{a_{f}(n_{0})}{n_{0}^{k/2}}\left(\dfrac{\abs{n}}{\delta(\af)}\right)^{k/2}\,\prod_{p\mid N}W_{p}(a(n/\delta(\af))\sigma).
\end{equation*}
To handle the terms $W_{p}(a(n/\delta(\af))\sigma)$ for $p\mid N$ -- and to suggest how one might go about evaluating them -- we invoke the Bruhat decomposition of $\sigma$ which, with notation as in \S \ref{sec:fourier_classical}, is given by
\begin{equation}
\sigma=\arraycolsep=3pt\def\arraystretch{0.7}\begin{pmatrix}
d&-b\\-q&a
\end{pmatrix}=z(q)n(-d/q)a(1/q^{2})wn(-a/q) \label{eq:bruhat_for_sacling}
\end{equation}
so that
\begin{equation*}
W_{p}\left(a\left(\dfrac{n}{\delta(\af)}\right)\sigma\right)=\omega_{\chi,p}(q)\,\psi_{p}\left(\dfrac{-nd}{\delta(\af)q}\right)\,W_{p}\left(a\left(\dfrac{n}{q^{2}\delta(\af)}\right)w\,n\left(\dfrac{-a}{q}\right)\right).
\end{equation*}
We now consider the expression prime-by-prime; to this end write $q=\prod_{p\mid N}p^{q_{p}}$, $M=\prod_{p\mid N}p^{M_{p}}$, $N=\prod_{p\mid N}p^{N_{p}}$ and recall that\footnote{While $\delta(\af)$ depends only on $(N,q)$, the exponent $d_{\pi_{p}}(q_{p})$ depends on $q$ in its full.}
\begin{equation*}
\delta(\af)=\dfrac{[q^{2},Mq,N]}{q^{2}}=\dfrac{\prod_{p\mid N}p^{d_{\pi_{p}}(q_{p})}}{q^{2}}
\end{equation*}
where we define
\begin{equation}\label{eq:d_p-def}
d_{\pi_{p}}(q_{p}):=\max\{2q_{p}, M_{p}+q_{p}, N_{p}\}.
\end{equation}
Then, for some $p\mid N$, the right-$K_{p,1}(N)$ invariance of $W_{p}$ implies

\begin{equation}\label{eq:the_prop_proof}
\begin{array}{>{\ds}l}\vspace{0.1in}
W_{p}\left(a\left(\dfrac{n}{q^{2}\delta(\af)}\right)w\,n\left(\dfrac{-a}{q}\right)\right)\\
\hspace{1in}=\,\omega_{\chi,p}\left(\frac{n(q^2\delta(\af),p^{\infty})}{(n,p^{\infty})q^2\delta(\af)}\right)W_{p}\left(a(p^{n_{p}-d_{\pi_{p}}(q_{p})})wn\left(u_{p} p^{-q_{p}}\right)\right)
\end{array}
\end{equation}
where $u_{p}=-a\dfrac{(n,p^{\infty})\delta(\af)q}{n(\delta(\af)q,p^{\infty})}\in\Zpx$. A convenient notation for us now shall be
\begin{equation}\label{eq:gtlv}
g_{t,l,v}=: a(p^{t})wn(vp^{-l})=
\begin{pmatrix}
p^{t}&\\&1
\end{pmatrix}
\begin{pmatrix}
&1\\-1&
\end{pmatrix}
\begin{pmatrix}
1&vp^{-l}\\&1
\end{pmatrix}
\end{equation}
for $t,l\in\Z$ and $v\in\Zpx$. This notation is in accordance with \cite{As17,As18,corbett-saha,saha-hybrid,saha-large-values} and we use it here to summarise our computations with the following proposition.

\begin{prop} \label{prop:the_proposition}
Fix a cusp $\af=\sigma^{-1}\infty$ of denominator $q=\prod_{p\mid q}p^{q_{p}}$, writing $\sigma^{-1}=\left(\begin{smallmatrix}
a&b\\q&d
\end{smallmatrix}\right)$
with $(a,N)=1$. 
Let $M=\prod_{p\mid N}p^{M_{p}}$ denote the conductor of a Dirichlet character $\chi$ let the level of $f$ be denoted by $N=\prod_{p\mid N}p^{N_{p}}$. Consider an arbitrary integer $n=n_{0}\prod_{p\mid N}p^{n_{p}}$, where $(n_{0},N)=1$, and for each $p\mid N$ define
\begin{equation*}
u_{p}:=-a\times\dfrac{p^{n_{p}}}{n}\times\dfrac{q\delta(\af)}{p^{d_{\pi_{p}}(q_{p})-q_{p}}}\in\Zpx
\end{equation*}
with $d_{\pi_{p}}(q_{p})$ as in \eqref{eq:d_p-def}. Then the $n$-th Fourier coefficient $a_{f}(n;\af)$ in the expansion of a Hecke newform $f$ at the cusp $\af$ is given by the formula
\begin{equation*}
\dfrac{a_{f}(n;\af)}{n^{k/2}}=\Omega_{\chi,q,\delta(\af)}\dfrac{a_{f}(n_{0})}{n_{0}^{k/2}}\,e\left(\dfrac{nd\,\overline{q/(q,N^{\infty})}}{\delta(\af)(q,N^{\infty})}\right)\dfrac{\prod_{p\mid N}\omega_{\chi,p}(n_0)W_{p}(g_{n_{p}-d_{\pi_{p}}(q_{p}),q_{p},u_{p}})}{\delta(\af)^{k/2}}
\end{equation*}
where
\begin{equation}\label {eq:def_of_unitary_coeffi}
	\Omega_{\chi,q,\delta(\af)} := \prod_{p \mid N} \omega_{\chi,p}^{-1}\left(\frac{q\delta(\af)}{(q^2\delta(\af),p^{\infty})}\right)
\end{equation}
and $W_p$ are the normalised local Whittaker new vectors associated to $f$ as in \eqref{eq:factorisation_of_W_varphi}.
\end{prop}

\begin{proof}
After the deduction of \eqref{eq:the_prop_proof}, the proof now follows from the observation that, as $-nd[q/(q,N^{\infty})]^{-1}/\delta(\af)(q,N^{\infty})\in\Qx$, we trivially have
\begin{equation*}
\prod_{p\mid N}\psi_{p}\left(-n\frac{d[q/(q,N^{\infty})]^{-1}}{\delta(\af)(q,N^{\infty})}\right)=e\left(\dfrac{nd\,\overline{q/(q,N^{\infty})}}{\delta(\af)(q,N^{\infty})}\right)
\end{equation*}
since $\psi_{p}$ is trivial on $\Zp$. Moreover we have $d\equiv a^{-1} \Mod{q}$.
\end{proof}

\begin{rem}
By Proposition \ref{prop:the_proposition}, we reduce the problem of understanding the coefficients $a_{f}(n;\af)$ to the local problem of evaluating the terms $W_{p}(g_{t,l,v})$ for $p\mid N$. There is a quite beautiful method to do so via taking the Fourier expansion of the function $v\mapsto W_{p}(g_{t,l,v})$ on $\Zpx/(1+p^{N_p}\Z_p)$ and computing the Fourier coefficients via the Jacquet--Langlands' local functional equation for $\GL_{2}$ (see \cite[\S 2]{saha-large-values}). These computations were performed to some extent in \cite{As17} and then in greater detail in the forthcoming (University of Bristol) PhD thesis of the first named author.
\end{rem}

\begin{rem}
The \textit{twist} of $f$ by a primitive Dirichlet character $\xi$ is given by
\begin{equation}
	[\xi f](z) := \sum_{n\geq 1} \xi(n) a_f(n)e(nz). \nonumber
\end{equation}
Whilst $\xi f$ may not be a newform itself (even if $f$ is), there exists a unique newform $f^{\xi}$ which generates the automorphic representation $\omega_{\xi}\pi$. From the identification of the associated local Whittaker newforms $W_p$ as in \eqref{eq:factorisation_of_W_varphi}, the Fourier expansion of $f^{\xi}$ at the cusp infinity is easily determined by means of \eqref{eq:coeffs_infty_muliplicative}, say. However, at other cusps $\af\neq\infty$ this is not the case. Proposition \ref{prop:the_proposition} shows that $a_{f^{\xi}}(n;\af)$ can be expressed in terms of the Whittaker new vectors in $\Wh(\omega_{\xi,p}\pi_p ,\psi_p)$, however at least when $\omega_{\xi}$ ramifies at $p$,
there is no uniform formula relating $W_{p}$ and $\xi f$.

\end{rem}

\subsection{Bounds for Fourier coefficients on average}

We now explore some amusing corollaries to Proposition \ref{prop:the_proposition}.


\begin{cor}\label{cor:ppprrp}
Let $0\leq \theta\leq\frac{7}{64}+\epsilon$ be an exponent towards the Ramanujan conjecture for Hecke eigenvalues (see \cite{Blomer-Brumley}). Then we have
\begin{equation}
	\sum_{0<n\leq X} \vert a_f(n;\af)\vert^2 \ll \frac{(k+1)^2}{\delta(\af)^k}\left( X^{k+2\theta} + X^{k-\frac{1}{2}}(q,{N}/{q})^{k+2\theta}\right).\nonumber
\end{equation}
\end{cor}

\begin{proof}
According to Proposition~\ref{prop:the_proposition} we have to bound
\begin{align}
	\sum_{0<n\leq X} \vert a_f(n;\af)\vert^2 = \sum_{n_1\mid N^{\infty}} \left(\frac{n_1}{\delta(\af)}\right)^{k}\sum_{\substack{0<n_0\leq \frac{X}{n_1}\\ (n_0,N)=1}} \vert a_f(n_0)\vert^2 \prod_{p \mid N} \vert W_p(g_{n_{1,p}-d_{\pi_{p}}(q_{p}),q_{p},n_0^{-1}u_{p}'})\vert^2. \nonumber
\end{align}
Observe that the function 
\begin{equation}
	n_0 \to \prod_{p \mid N} \vert W_p(g_{n_{1,p}-d_{\pi_{p}}(q_{p}),q_{p},n_0^{-1}u_{p}'})\vert \nonumber
\end{equation}
is $(q,{N}/{q})$ periodic, see also \cite[Lemma~3.12]{saha-hybrid}. Thus, by the Chinese reminder theorem and \cite[Proposition~2.10]{saha-hybrid} we have  
\begin{multline}
	\sum_{n_0=r(q,{N}/{q})}^{(r+1)(q,\frac{N}{q})-1} \prod_{p \mid N} \vert W_p(g_{n_{1,p}-d_{\pi_{p}}(q_{p}),q_{p},n_0^{-1}u_{p}'})\vert^2 \\ =  (q,N/q) \prod_{p \mid N} \int_{\Zpx}\vert W_p(a(v)g_{n_{1,p}-d_{\pi_{p}}(q_{p}),q_{p},u_p'})\vert^2d^{\times}v\ll \frac{(q,{N}/{q})}{n_1^{\frac{1}{2}}}.\nonumber
\end{multline}
for any $r\in \N_0$. Thus, applying the bound $a_f(n_0)\ll n_0^{\theta}$ yields
\begin{equation}
	\sum_{0<n\leq X} \vert a_f(n;\af)\vert^2 = \frac{(q,{N}/{q})^{k+2\theta}}{\delta(\af)^k}\sum_{\substack{n_1\mid N^{\infty},\\n_1\ll X}} n_1^{k-\frac{1}{2}}\sum_{0\leq r\leq \lceil\frac{X}{n_1(q,{N}/{q})}\rceil} (r+1)^{k-1+2\theta}.\nonumber
\end{equation}
The result follows after estimating the remaining sums, in particular the $n_1$-sum using the Rankin trick.
\end{proof}

\begin{rem}
While this bound  is satisfactory for holomorphic modular forms, in the Maa\ss\  case, where the Ramanujan conjecture is still open, this fails to provide the expected upper bound for a second moment.
Another approach to bounding averages of Fourier coefficients is to use the bound \cite[Corollary~2.35]{saha-large-values} and estimate the ramified coefficients globally. This approach yields
\begin{align}
	\sum_{0< n\leq X} \vert a_f(n;\af)\vert^2 \ll N^{\frac{1}{2}+\epsilon}\sum_{n_1\mid N^{\infty}} \left(\frac{n_1}{\delta(\af)}\right)^{k}\sum_{\substack{n_0\leq \frac{X}{n_1}\\ (n_0,N)=1}} \vert a_f(n_0)\vert^2\ll C(f)^{\epsilon} \frac{N^{\frac{1}{2}}X^{k+\epsilon}}{\delta(\af)^k} . \nonumber
\end{align}
\end{rem}

\subsection{A worked example: principle series}

Let $f$ be a holomorphic modular form of nebentypus $\chi$. Furthermore, assume that $M=N=p^{h}$ for even $h$. Thus, $\pi_{p}$ is the irreducible principle series $\omega_{\chi,p}\vert \cdot \vert^{s}\boxplus \vert \cdot \vert^{-s}$ in the sense of \cite{jacquet-langlands}  (see also \cite[\S 2.1.6]{corbett-saha} for notation). In this specific case one can extract a precise expression for $W_{p}$ from \cite[Lemma~5.8]{As17} for the sake of completeness we will give a proof here. 
\begin{lem}\label{lem:claim}
We have
\begin{equation}
	W_{p}(g_{-\frac{3h}{2},\frac{h}{2},u_p}) = \omega_{\chi,p}(-u_p)\psi(-u_p^{-1}p^{-h})
	p^{\frac{h}{4}-\frac{sh}{2}} \delta(u_p\in-b_{\chi}^{-1}+p^{\frac{h}{2}}\Z_p). \nonumber
\end{equation}
\end{lem}
\begin{proof}

For the proof of this lemma we will temporarily introduce some notation from \cite{As17, saha-large-values}. Indeed, let $W_{\pi_p}$ be the unique $K_{1,p}(N)'$-fixed vector normalised by $W_{\pi_p}(1)=1$. Using \cite[Lemma~2.18, Propositon~2.28]{saha-large-values} we find that  
\begin{equation}
	W_{p}(g_{-\frac{3h}{2},\frac{h}{2},u_p}) = \varepsilon({1}/{2},\pi_p)\omega_{\chi,p}(u_p)\psi(-u_p^{-1}p^{-h})W_{\pi_p}(g_{-\frac{3h}{2},\frac{h}{2},-u_p}). \nonumber
\end{equation}
Furthermore according to \cite[(11)]{saha-large-values} we expand
\begin{equation}
	W_{\pi_p}(g_{-\frac{3h}{2},\frac{h}{2},-u_p}) = \sum_{\substack{a(\mu)\leq \frac{h}{2},\\ \mu(p)=1}} c_{-\frac{3h}{2},\frac{h}{2}}(\mu)\mu(-u_p). \nonumber
\end{equation}
Fortunately the finite Fourier coefficients have been computed in \cite[Lemma~2.3]{As17}. With this at hand we obtain
\begin{equation}
	W_{\pi_p}(g_{-\frac{3h}{2},\frac{h}{2},-u_p}) = \varepsilon({1}/{2},\pi_p)\psi(-u_p^{-1}p^{-h})\frac{p^{\frac{sh}{2}-\frac{h}{4}}}{1-p^{-1}}\sum_{\substack{a(\mu)\leq \frac{h}{2},\\ \mu(p)=1}}[\omega_{\chi,p}\mu](u_p)\varepsilon({1}/{2},\mu^{-1}\omega_{\chi,p}^{-1}). \nonumber
\end{equation}
We now apply \cite[Lemma~2.37]{saha-large-values} to find $b_{\chi}\in\Zpx$ such that
\begin{align}
	W_{\pi_p}(g_{-\frac{3h}{2},\frac{h}{2},-u_p}) = \omega_{\chi,p}(-u_p)\psi(-u_p^{-1}p^{-h})\frac{p^{\frac{-sh}{2}-\frac{h}{4}}}{1-p^{-1}}\sum_{\substack{a(\mu)\leq \frac{h}{2},\\ \mu(p)=1}}\mu(-u_pb_{\chi}).\nonumber
\end{align}
The claim lemma now follows directly from orthogonality of characters.
\end{proof}

Together with Proposition \ref{prop:the_proposition}, Lemma \ref{lem:claim} implies
\begin{equation}
	a_f(n,{a}/{p^{\frac{h}{2}}}) = \begin{cases}
		\chi(a)p^{-\frac{sh}{2}} a_f(n)N^{-\frac{k-1}{4}} &\text{ if } (n,p)=1 \text{ and }n\equiv ab_{\chi} \mod p^{\frac{h}{2}},\\
		0 &\text{ else}
	\end{cases}
\end{equation}
remarking that $\abs{\chi(a)p^{-\frac{sh}{2}}}=1$. Furthermore, Corollary \ref{cor:ppprrp} gives the bound
\begin{equation}
	\sum_{0<n\leq X} \vert a_f(n;{a}/{p^{\frac{h}{2}}})\vert^2 \ll N^{\epsilon}k^2X^{k+\epsilon}\left(N^{-\frac{k}{2}}+X^{-\frac{1}{2}}\right). \nonumber
\end{equation}
Note that individually $\abs{a_f(n,{a}/{p^{\frac{h}{2}}})}^2 \ll_{\epsilon} n^{k-1+\epsilon}N^{-\frac{k}{2}+\frac{1}{2}}$. Thus for large $X$ the average bound above exploits the narrow support of the coefficients $a_f(n,{a}/{p^{\frac{h}{2}}})$.

\section{Generalised Atkin--Lehner relations}\label{sec:atkin_lehner}

In this section we will show how to relate the Fourier coefficients at different cusps. Classically, this corresponds to the Atkin--Lehner involution. From our point of view it will appear in terms of certain `functional equations' of the local Whittaker functions $W_p$. The upshot is that we get a large variety of relations between the cusps. Even though these relations become combinatorially quite involved and require a bit of notation we will maintain a description of the complete picture.

We start by providing the key identities. Indeed, this section is based on the following three lemmata. The first one is a local incarnation of the fact that cusps of the form ${a}/{q}$ with $N\mid q$ are equivalent to the cusp $\infty$. 

\begin{lem}\label{lm:second_AL}
Let $l \geq N_p$, then we have
\begin{equation}
	W_{p}(g_{t,l,v}) = \psi_p(-v^{-1}p^{t+l})\omega_{\chi,p}(-vp^{-l})p^{-\frac{t+2l}{2}}\lambda_{f}(p^{t+2l})\nonumber
\end{equation}
for all $v\in \Z_p^{\times}$ and all $t\in\Z$. 
\end{lem}
This is a reformulation of \cite[Lemma~2.3]{As18}. We include a shorter more instructive proof here.
\begin{proof}
Note that
\begin{equation}
\begin{array}{rcl}\vspace{0.1in}
	g_{t,l,v}=\left(\begin{matrix} 0&p^t \\ -1 & -vp^{-l}	\end{matrix}\right) &=& \left(\begin{matrix} p^{t+l} & -v^{-1}p^t \\ 0 & p^{-l}	\end{matrix}\right) \left(\begin{matrix} -v^{-1} &  0 \\ -p^{l} & -v	\end{matrix}\right) \\ &=& z(-vp^{-l})n(-v^{-1}p^{t+l})a(p^{t+2l}) \left(\begin{matrix} v^{-2} &  0 \\ v^{-1}p^{l} & 1	\end{matrix}\right). \nonumber
\end{array}
\end{equation}
Since $l\geq N_p$, the final matrix above is in $K_{1,p}(N)$. Thus, exploiting the transformation behaviour of $W_{p}$ and \eqref{eq:shintani} yields the result at once.
\end{proof}

The next result is a local take on switching the cusps $0$ and $\infty$ using the classical Atkin--Lehner involution. 
\begin{lem} \label{lm:first_atkin_lehner}  
For all $v\in\Z_p^{\times}$ we have
\begin{equation}
	W_{p}(g_{t,0,v}) = 	\varepsilon({1}/{2},\pi_p)p^{-\frac{t+N_p}{2}}\lambda_{\tilde{\pi}}(p^{t+N_p}),  \nonumber	
\end{equation}
\end{lem}
The result can be derived from \cite[\S 2.2.2]{As18} together with \cite[Corollary 2.26]{saha-large-values}. Let us give a slightly more direct proof here.
\begin{proof}
We put $W'=\pi_{p}(w)W_{p}$ and define
\begin{equation}
	Z_p(W',s,\mu) = \int_{\Q_p^{\times}}W'(a(y))\mu(y)\abs{y}_p^{s-\frac{1}{2}}d^{\times}y.
\end{equation} 
The local functional equation for this zeta integral is given by
\begin{equation}
	Z_p(W', {1}/{2}+it,\mu) = \frac{L(\frac{1}{2}+it,\mu\pi_p)Z_p(\pi_p(w)W',\frac{1}{2}-it,\mu^{-1}\omega_{\chi,p}^{-1})}{\varepsilon({1}/{2}+it,\mu\pi_p)L({1}/{2}-it,\mu^{-1}\tilde{\pi}_p)}. \nonumber
\end{equation} 
We can rewrite this as
\begin{equation}
\begin{array}{rc>{\ds}l}\vspace{0.1in}
	Z_p(W',{1}/{2}+is,\mu) &=&\frac{L(\frac{1}{2}+is,\mu\pi_p)Z_p(\pi_p(z(-1))W_{p},{1}/{2}-is,\mu^{-1}\omega_{\chi,p}^{-1})}{\varepsilon({1}/{2}+is,\mu\pi_p)L({1}/{2}-is,\mu^{-1}\tilde{\pi}_p)} \nonumber \\
	&=& \delta(\mu\vert_{\Zpx}=\omega_{\chi,p}\vert_{\Zpx}) \varepsilon({1}/{2}-is,\pi_{p})L({1}/{2}+is,\tilde{\pi}_{p}). \nonumber
\end{array}
\end{equation}
By $p$-adic Mellin inversion we find
\begin{align}
	W_{p}(g_{t,0,v})  = W'(a(p^t)) &= \frac{\log(p)}{2\pi}\sum_{\mu}\int_{-\frac{\pi}{\log(p)}}^\frac{\pi}{\log(p)}Z(W',\frac{1}{2}+is,\mu)p^{ist}ds\nonumber \\
	&= \varepsilon({1}/{2},\pi_{p})\frac{\log(p)}{2\pi}\int_{-\frac{\pi}{\log(p)}}^\frac{\pi}{\log(p)}L(\frac{1}{2}+is,\tilde{\pi}_{p})p^{is(t+N_p)}ds. \nonumber
\end{align}
The statement follows after evaluating the remaining archimedean integral.
\end{proof}

The final lemma lies at the heart of the generalised Atkin--Lehner involution. 
\begin{lem} \label{lm:third_aL}
Let $0\leq l\leq v_p(N)$. Then we have
\begin{equation}
	W_{p}(g_{t,l,v})= \varepsilon({1}/{2},\pi_p)\psi_p(-p^{t+l}v^{-1})\omega_{\chi,p}(vp^{t+l})\tilde{W}_{p}(g_{t+2l-N_p,N_p-l,-v}). \nonumber
\end{equation}
Here $\tilde{W}_p$ is the normalised new vector in $\tilde{\pi}_p$.
\end{lem} 
This is exactly \cite[Proposition~2.28]{saha-large-values} after accounting for an unramified twist to remove the assumption $\omega_{\chi,p}(p)=1$. To avoid shuffling around the unramified twist we give a self-contained proof.
\begin{proof}
Let $W_{p}'$ be the unique $K_{1,p}'(N)$-fixed vector normalised by $W_{p}'=1$. We write $$w_{N_p}:=z(p^{N_p})a(p^{-N_p})w$$ and check that, according to Lemma~\ref{lm:first_atkin_lehner}, we have 
\begin{equation}
	W_{p}(w_{N_p}) = \omega_{\chi,p}(p^{N_p})\varepsilon({1}/{2},\pi_{p}). \nonumber
\end{equation}
In particular, after checking the right transformation behaviour, we find that $$\pi_p(w_n)W_{p} = \omega_{\chi,p}(p^{N_p})\varepsilon({1}/{2},\pi_{p})W_{p}'.$$ For $0\leq l\leq N_p$ we have the matrix identity
\begin{equation}
	g_{t,l,v} = n(-v^{-1}p^{l+t})z(vp^{-l})g_{t+2l-N_p,N_p-l,-v}\left(\begin{matrix}1 & 0 \\ 0 & v^{-2}
	\end{matrix}\right)w_{N_p}. \nonumber
\end{equation}
We compute
\begin{equation}
		W_{p}(g_{t,l,v}) = \varepsilon({1}/{2},\pi_{p})\omega_{\chi,p}(vp^{N_p-l})\psi_p(-v^{-1}p^{t+l})W_{p}'(g_{t+2l-N_p,N_p-l,-v}).\nonumber
\end{equation}

To conclude the proof we must identify $W_{p}'$ in terms of $\tilde{W}_{p}$. But given the isomorphism
\begin{equation}
\Psi\colon\Wh(\pi_p,\psi_p) \to \Wh(\tilde{\pi}_p,\psi_p), \quad W\mapsto [\omega_{\chi,p}^{-1}\circ\det]\cdot W \nonumber 
\end{equation}
and since
\begin{equation}
	\Psi(W_{p}') \vert_{K_{1,p}} = 1, \nonumber
\end{equation}
the uniqueness of the new vector in $\tilde{\pi}_p$ implies that
\begin{equation}
	\omega_{\chi,p}(\det(g))^{-1}W_{p}'(g) = \tilde{W}_{p}(g). \nonumber
\end{equation}
\end{proof}

We now piece together these local results. Given a cusp ${a}/{q}$, we can flip the sign of $v_p(q)$ according to the lemmata above at any place $p\mid N$. Let us moreover assume the cusp to be in standard form; that is, $q\mid N$. For a set $S\subset \{ p \mid N \}$ define
\begin{equation} 
	q^S:= \prod_{p\in S} p^{N_p-q_p} \prod_{\substack{p \mid N,\\ p\nmid S}} p^{q_p}\quad \text{and}\quad M^S :=  \prod_{p\in S}p^{M_p}.
\end{equation}
Factorise the global character $\omega_{\chi} = \prod_{p\mid M} \omega_{\chi}^{(p)}$. Here, the $ \omega_{\chi}^{(p)}$ are Hecke characters of conductor $p^{M_p}$, so that their $p$th-component restricted to $\Zpx$ equals the $p$th-component of $\omega_{\chi}$ restricted to $\Zpx$ and $\omega_{\chi,p}^{(p)}(p) = 1$. (This is an elaborate way of stating that, if $\chi$ factors into a product of primitive Dirichlet characters $\chi^{(p)}$ of conductor $p^{v_p(M)}$, then $\omega_{\chi}^{(p)}$ is the adelisation of $\chi^{(p)}$.) With this at hand we define
\begin{equation}
	\omega_{\chi}^S = \prod_{p\mid M^S}  [\omega_{\chi}^{(p)}]^{-1}\quad \text{and}\quad \pi^{S} = \omega_{\chi}^S  \pi. \nonumber
\end{equation}
The idea above is to replace $\pi_p$ (up to unramified twist) with $\tilde{\pi}_p$ at each place $p\mid M^S$. Note that the conductor of $\omega_{\pi}^S$ is $M^S$, while the conductor of $\pi^S$ remains $N$. Finally, let $f^S$ be the newform associated to $\pi^S$. This turns out to be the correct (generalised) Atkin--Lehner partner of $f$ for given $S$. The generalised (pseudo)-Atkin--Lehner eigenvalue will turn out to be related to
\begin{equation}
	\eta_{\af}(f,S) := \prod_{p \in S } \varepsilon({1}/{2},\pi_p)\omega_{\chi,p}(-a). \nonumber
\end{equation}
We are now ready to establish the following proposition.

\begin{prop}
Let $S$ be as above. Then we have
\begin{equation}
	a_f(n;\af) = \eta_{\af}(f,S)\left[ \prod_{p \in S }\psi_p\left( \frac{a^Sn}{q^S\delta(\af^S)} \right)\omega_{\chi,p}^{-1}(q^S) \right] \left(\prod_{p \in S } p^{q_p^S-q_p}\right)^{\frac{k}{2}} a_{f^S}(n,\af^S). \nonumber
\end{equation}
The cusp $\af^S = {a^S}/{q^S}$, where $a^S$ is uniquely determined modulo $\delta(\af)q = \delta(\af^S)q^S$ defined by
\begin{equation}
	a^S \equiv \begin{cases} 
		a \mod p^{d_{\pi_p}(q_p)-q_p} &\text{ if $p\mid N$ and $p\not\in S$},\\
		-a \mod p^{d_{\pi_p}(q_p^S)-q_p^S} &\text{ if $p\in S$.}
	\end{cases}	\nonumber
\end{equation}
\end{prop}
\begin{proof}
We start with Proposition~\ref{prop:the_proposition} and inspect the product of local Whittaker functions piece by piece. At  each place $p\in S$ we apply Lemma~\ref{lm:third_aL} and find
\begin{multline}
	W_{p}(g_{n_p-d_{\pi_p}(q_p),q_p,u_p}) = \varepsilon({1}/{2},\pi_p)\psi_p\left( \frac{a^{-1}n q}{[q^2,qM,N]}\right)\omega_{\chi,p}\left(u_pp^{n_p-d_{\pi_p}(q_p)+q_p}\right) \\ \times\, \tilde{W}_{p}(g_{n_p-d_{\pi_p}(N_p-q_p),N_p-q_p,-u_p}). \nonumber
\end{multline}
At the remaining places, $p\not\in S$, we leave the value of $W_p$ as it is. Note that $q_p-\max(2q_p,M_p+q_p, N_p)$ is invariant under changing $q_p$ to $N_p-q_p$. We arrive at
\begin{align} 
	a_f(n;\af) =& \eta_{\af}(f,S)\Omega_{\chi,q,\delta(\af)} a_f(n_0) \left(\frac{n_1}{\delta(\af)}\right)^{\frac{k}{2}} 
	\prod_{p\notin S}\psi_p\left( -\frac{a^{-1}n q}{[q^2,qM,N]}\right) \nonumber \\
	&\times\, \prod_{p \notin S, p\mid N }\omega_{\chi,p}\left(\frac{n}{p^{n_p}}\right)W_p(g_{n_p-d_{\pi_p}(q_p),q_p,u_p}) \nonumber \\
	&\times\, \prod_{p \in S }\omega_{\chi,p}\left(\frac{q\delta(\af)}{p^{2d_{\pi}(q_p)-2q_p-n_p}}\right)\tilde{W}_p(g_{n_p-d_{\pi_p}(N_p-q_p),N_p-q_p,-u_p}).\nonumber
\end{align}
We make two observations. Firstly, $\delta(\af)^{-1} = \frac{q}{q^S\delta(\af^S)}$. Secondly,
\begin{equation}
 a_f(n_0)	\prod_{p \in S }\omega_{\chi,p}(n_0)=a_{f^S}(n_0). \nonumber
\end{equation}
Finally, writing $W^S_p$ for the $p$-Whittaker new vector of $\pi^S$, we have
\begin{equation}
	W_p^S(g_{n_p-d_{\pi_p}(q_p),q_p,u_p}) = W_p(g_{n_p-d_{\pi_p}(q_p),q_p,u_p})  \prod_{l\in S}\omega_{\chi,l}(p^{n_p-d_{\pi_p}(q_p)}) \nonumber
\end{equation}
for $p\mid N$ and $p\notin S$. On the other hand, if $p\in S$, then
\begin{multline}
	W_p^S(g_{n_p-d_{\pi_p}(N_p-q_p),N_p-q_p,-u_p}) = \\\tilde{W}_p(g_{n_p-d_{\pi_p}(N_p-q_p),N_p-q_p,-u_p}) \prod_{l\in S}\omega_{\chi,l}(p^{n_p-d_{\pi_p}(N_p-q_p)}). \nonumber
\end{multline}
Summing up what we have observed so far,
\begin{align}
	a_f(n;\af) =& \eta_{\af}(f,S)\Omega_{\chi,q,\delta(\af)}\left(\frac{q}{q^S}\right)^{\frac{k}{2}} \prod_{p \in S }\omega_{\chi,p}\left(\frac{q^2q^S\delta(\af)^2}{p^{2d_{\pi_p}(q_p)-2q_p}} \right)   \prod_{p\notin S}\psi_p\left( -\frac{a^{-1}n q}{[q^2,qM,N]}\right) \nonumber \\
		&\times\, a_{f^S}(n_0) \left(\frac{n_1}{\delta(\af^S)}\right)^{\frac{k}{2}} \prod_{p\mid N }[\omega_{\chi,p}(\omega_{\chi,p}^S)^2]\left(\frac{n}{p^{n_p}}\right)W_p^S(g_{n_p-d_{\pi_p}(q_p^S),q_p^S,u_p^S}). \nonumber 
\end{align}
Finally, we have
\begin{equation}
	\varepsilon_q\varepsilon_{q^S}^{-1} = \prod_{p \in S }\omega_{\chi,p}^{-1}\left(\frac{q^2(q^S)^2\delta(\af)^2}{p^{2d_{\pi_p}(q_p)-2q_p}} \right).\nonumber
\end{equation}
Thus, using Proposition~\ref{prop:the_proposition} for $a_{f^S}(n;\af^S)$, yields the desired result.
\end{proof}

\section{The Vorono\"i summation formulae}\label{sec:voronoi}

We now discuss some variants of the Vorono\"i summation formula for $\GL_{2}$. Our starting point is the following pre-Vorono\"i formula.

\begin{lem} \label{lm:basic_whitt_id}
Let $\zeta\in \A$ and let $\phi$ be a cuspidal automorphic form in the representation space of $\pi$. Then we have
\begin{equation}
	\sum_{\xi\in\Q^{\times}} \psi(\xi\zeta)W_{\phi}\left(\left(\begin{matrix} \xi&0 \\ 0&1 \end{matrix}\right)\right) = \sum_{\xi\in \Q^{\times}}\hat{W}_{\phi}\left(\left(\begin{matrix} \xi&0 \\ 0&1 \end{matrix}\right)\left(\begin{matrix} 1&0 \\-\zeta&1 \end{matrix}\right)\right) \label{eq:basic_voronoi}
\end{equation}
where $\hat{W}_{\phi}(g) := W_{\phi}(w{}^tg^{-1})$.
\end{lem}

The trick used alongside the pre-Voron\"i summation formula is to choose the cuspidal function in various ways $\phi$ to determine a variety of Vorono\"i summation formulae.
The following one is our take on a commonly used classical formula. This previously appeared in \cite[Lemma~2.1]{As18} and is based on \cite[Theorem~3.1]{Te_vor}.

\begin{theorem} \label{th:voronoi2}
Let $f$ be a modular form and let $F\colon \R_{>0}\to \R$ be a compactly supported smooth function and $\af = {p}/{q}$ be a cusp given by scaling matrix 
\begin{equation}
	\sigma_{\af} = \left( \begin{matrix} r& -s \\ -q & p \end{matrix} \right). \nonumber
\end{equation}
Then
\begin{multline}
	\sum_{\n\in\Nat} e\left(n\frac{a}{b}\right)a_f(n;\af) \left(\frac{n}{\delta(\af)}\right)^{-\frac{k-1}{2}} F\left(\frac{n}{\delta(\af)}\right) \\ = \frac{\delta(\af)^{\frac{1}{2}}}{b} \sum_{ n\in\Z_{\neq 0}}e\left(-n\frac{\overline{a\delta(\af)/(\delta(\af),b)}}{\delta(\mathfrak{b})b(\delta(\af),b)^{-1}}\right)a_f(n;\mathfrak{b})\left(\frac{n}{\delta(\mathfrak{b})}\right)^{-\frac{k-1}{2}}[\mathcal{H}_f F]\left(\frac{n}{\delta(\mathfrak{b})b^2} \right) \nonumber
\end{multline}
where the cusp $\mathfrak{b} := {(ap\delta(\af)+bs)}/{(qa\delta(\af)+rb)}$ is not written in lowest terms, and the Hankel transform $\mathcal{H}_f F$ of $F$ is defined by 
\begin{equation}\label{eq:def_hankel_transform_hol}
	\mathcal{H}_f F(y)  :=
			\delta_{y>0}2\pi i^k \int_0^{\infty} J_{k-1}\left(4\pi\sqrt{y x}\right)F(x)dx
\end{equation}			
if $f$ is holomorphic and by
\begin{equation}\label{eq:def_hankel_transform_maa}
	\mathcal{H}_f F(y)  :=
			\begin{cases}\vspace{0.1in}
				\frac{\pi i}{\sinh(\pi t_f)} \int_0^{\infty} (J_{2it_f}\left(4\pi\sqrt{y x}\right)-J_{-2it_f}\left(4\pi\sqrt{y x}\right))F(x)dx&\text{if $y>0$}\\
				4\cosh(\pi t_f) \int_0^{\infty} K_{2it_f}\left(4\pi\sqrt{\abs{y} x}\right)F(x)dx&\text{if $y<0$.}
			\end{cases}
\end{equation}
if $f$ is Maa\ss\ form.
\end{theorem}
\begin{proof}
We will start by setting up the left hand side of \eqref{eq:basic_voronoi} accordingly. We choose $\zeta$ to be $-\frac{a}{b}$ embedded into $\A$ via
\begin{align}
	\Q \emblong \A_{\f } \emblong \A. \nonumber  
\end{align}
Furthermore we pick $\phi$ as follows. Write $\phi = (v_{\phi,\f }, v_{\phi,\infty})$ under the isomorphism $\pi = \pi_{\f }\otimes \pi_{\infty}$. Then $\phi$ will  be uniquely determined by assuming that $v_{\phi,\f }=\pi_{\f}(a(\delta(\af)^{-1})\sigma_{\af})v_{\varphi,\f }$ and that in the Kirillov model of $\pi_{\infty}$ we have $ v_{\phi,\infty} = \pi_{\infty}(a(\delta(\af)^{-1})[\abs{\cdot}^{\frac{1}{2}}\cdot F]$. With this at hand we find that
\begin{equation}
	W_{\phi}(a(\xi)) = W_{\phi,\infty}(a(\xi)) \cdot \prod_{p< \infty} W_{\phi,p}(a(\xi)). \nonumber
\end{equation}
Furthermore $\phi$ is chosen such that $W_{\infty,\phi}(a(\xi)) = \abs{\xi\delta(\af)^{-1}}^{\frac{1}{2}}F(\xi\delta(\af)^{-1})$ and
\begin{equation}
	\prod_{p<\infty} W_{\phi,p}(a(\xi)) = \prod_{p< \infty} W_{p}(a(\xi\delta(\af)^{-1})\sigma_{\af}) = \begin{cases}
		\left(\frac{n}{\delta(\af)}\right)^{-\frac{k}{2}} a_f(n;\af) &\text{ if $\xi=n\in \Nat$}, \\
		0 &\text{ else.}
	\end{cases}
\end{equation} 
as in \eqref{eq:factorisation_we_need}. These choices imply that
\begin{equation}
	\sum_{\xi\in\Q^{\times}} \psi(\xi\zeta)W_{\phi}\left(a(\xi)\right) = \sum_{\n\in\Nat} e\left(n\frac{a}{b}\right)a_f(n;\af) (n\delta(\af)^{-1})^{-\frac{k-1}{2}} F(n\delta(\af)^{-1}). \nonumber
\end{equation}

We now evaluate the right-hand side of \eqref{eq:basic_voronoi}. We do so by treating archimedean and finite places separately. This is possible as 
\begin{eqnarray}
	 \hat{W}_{\phi}\left(\left(\begin{matrix} \xi&0 \\ 0&1 \end{matrix}\right)\left(\begin{matrix} 1&0 \\-\zeta&1 \end{matrix}\right)\right) = \prod_{p\leq \infty} W_{\phi,p} \left( \left(\begin{matrix} \xi & 0 \\ 0 &1\end{matrix} \right)w\left(\begin{matrix} 1 & \zeta  \\ 0 &1\end{matrix} \right)\right). \nonumber
\end{eqnarray}
At $p=\infty$ we find that
\begin{equation}
	W_{\phi,\infty}(a(\xi)wn(\zeta_{\infty}))=  \sqrt{\xi}\delta(\af)[\mathcal{H}_f F](\xi\delta(\af)). \nonumber
\end{equation}
This holds since the action of $w$ in the Kirillov model is precisely described by this Hankel transform. See \cite{cogdell_BS} for details. 

In order to deal with the finite places we artificially write
\begin{equation}
	\delta = \left( \begin{matrix} \overline{a\delta(\af)(\delta(\af),b)^{-1}} & \frac{1-a\delta(\af)(\delta(\af),b)^{-1}\overline{a\delta(\af)(\delta(\af),b)^{-1}}}{b(\delta(\af),b)^{-1}} \\ -\frac{b}{(\delta(\af),b)} & \frac{a\delta(\af)}{(\delta(\af),b)} \end{matrix} \right) \in \SL_{2}(\Z) \nonumber
\end{equation}	 
and define $\mathfrak{b} = (\delta\sigma_{\af})^{-1}\infty$. Note that $\mathfrak{b} = \frac{ap\delta(\af)+bs}{qa\delta(\af)+rb}$ in non-lowest terms. We have set things up so that
\begin{multline}
	a(\xi)wn\left(-\frac{a}{b}\right)a(\delta(\af)^{-1})\sigma_{\af} =\\ z([\delta(\af),b]^{-1})a\left(\xi [\delta(\af),b]\frac{b}{(\delta(\af),b)}\right)n\left( \frac{\overline{a\delta(\af)(\delta(\af),b)^{-1}}}{b(\delta(\af),b)^{-1}}\right)\delta  \sigma_{\af}. \nonumber
\end{multline}
Thus we find that
\begin{multline}
	\prod_{p< \infty} W_{\phi,p} \left( a(\xi)wn(\zeta_{\f })\right) = \omega_{\pi,\f }([\delta(\af),b]^{-1}) \psi_{\f }\left(\xi[\delta(\af),b]\overline{a\delta(\af)(\delta(\af),b)^{-1}}\right)\\ \times\, \prod_{p< \infty} W_{\varphi,p} \left( a\left(\xi \frac{[\delta(\af),b]b}{(\delta(\af),b)}\right) \iota_{{\f}}(\delta\sigma_{\af})\right). \nonumber 
\end{multline}
At this stage we use \eqref{eq:factorisation_we_need} and obtain that
\begin{equation}
	\prod_{p< \infty} W_{\varphi,p} \left( a\left(\xi \frac{[\delta(\af),b]b}{(\delta(\af),b)}\right) \iota_{{\f}}(\delta\sigma_{\af})\right) = \begin{cases}
		\left(\frac{\delta(\mathfrak{b})}{n} \right)^{\frac{k}{2}} a_f(n;\mathfrak{b}) &\text{ if $\xi = \frac{n(\delta(\af),b)}{\delta(\mathfrak{b})[\delta(\af),b]b}$} \\
		0 &\text{ else.}
	\end{cases}
\end{equation}
Finally, we have all the ingredients together. We find that
\begin{multline}
	\sum_{\xi\in \Q^{\times}}\tilde{W}_{\phi}\left(\left(\begin{matrix} \xi&0 \\ 0&1 \end{matrix}\right)\left(\begin{matrix} 1&0 \\-\zeta&1 \end{matrix}\right)\right) =\\ 
	\delta(\af)\left(\frac{(\delta(\af),b)}{b[\delta(\af),b]}\right)^{\frac{k}{2}} \sum_{\xi\in\frac{(\delta(\af),b)}{\delta(\mathfrak{b})[\delta(\af),b]b}\Nat}e(-\xi[\delta(\af),b]\overline{a\delta(\af)(\delta(\af),b)^{-1}})\\
	\times\, a_f\left(\xi\delta(\mathfrak{b})\frac{[\delta(\af),b]b}{(\delta(\af),b)};\mathfrak{b}\right)\xi^{-\frac{k-1}{2}} [\mathcal{H}_f F](\xi\delta(\af)). \nonumber
\end{multline}
The statement follows in a straightforward manner after shifting the summation.
\end{proof}

\begin{cor} \label{cor:voronoi}
Let $f$ be a classical cusp form and let $F\colon \R_{>0}\to \R$ be a compactly supported smooth function. Then
\begin{multline}
	\sum_{\n\in\Nat} e\left(n\frac{a}{b}\right)a_f(n) n^{-\frac{k-1}{2}} F\left(n\right) = b^{-1} \sum_{0\neq n\in\Z}e\left(-n\frac{\overline{a}}{b\delta(\mathfrak{b})}\right)a_f(n;\mathfrak{b})\left(\frac{n}{\delta(\mathfrak{b})}\right)^{-\frac{k-1}{2}} \\ [\mathcal{H}_f F]\left(\frac{n}{[b^2, Mb, N]}\right) . \nonumber
\end{multline}
Where $\mathfrak{b}={a}/{b}$ and $\overline{a}$ is the inverse of $a$ modulo $b\delta(\mathfrak{b})$. 
\end{cor}
\begin{proof}
This is a direct consequence of Theorem~\ref{th:voronoi2} with $\af=\infty$. In particular we have $\sigma_{\af}=1$ and $\delta(\af)=1$.
\end{proof}

It is a nice exercise to apply Corollary~\ref{cor:voronoi} with $F(x) = x^{\frac{k-1}{2}}e(ixy)$ when $f$ is a holomorphic modular form of weight $k$. Note that even though the result is stated for compactly supported $F$ it is possible to extend its validity to the $F$ in question. According to \cite[6.631.(4)]{GR07} we find that
\begin{equation}
 2\pi i^k \sqrt{\xi}\int_0^{\infty}J_{k-1}(2\pi \sqrt{x\xi})F(x) dx=	i^k\xi^{\frac{k-1}{2}}y^{-k}e^{-2\pi \frac{\xi}{y}}. \nonumber
\end{equation}
Thus, the right-hand side of our Vorono\"i summation formula becomes
\begin{align}
	f\left(\frac{a}{b}+iy\right) &= \sum_{n\in\N}a_f(n)e\left(n\left(\frac{a}{b}+iy\right)\right) \nonumber \\
	 &= (-iby)^{-k}\sum_{n\in \N}a_f(n;\mathfrak{b})e\left( \frac{n}{\delta(\mathfrak{b})}\left(\frac{-\overline{a}}{b}+\frac{i}{yb^2}\right)\right) \nonumber \\
	 & = j(\sigma,\frac{a}{b}+iy)^{-k}f\vert_{k}\sigma^{-1}\left(\sigma\left(\frac{a}{b}+iy\right)\right)\nonumber\\
	 &=f\vert_{k}\sigma\sigma^{-1}\left(\frac{a}{b}+iy\right) = f\left(\frac{a}{b}+iy\right). \nonumber
\end{align} 
The last equality follows from $j(\sigma,\frac{a}{b}+iy) = -iby$ and $\sigma(\frac{a}{b}+iy) = -\frac{\overline{a}}{b}+\frac{i}{yb^2}$.

\begin{rem}
This example shows that Vorono\"i summation is precisely the transition between the Fourier expansions of $f$ at different cusps. While classically it is a very subtle matter to include an archimedean test function, it is straightforward adelically.
\end{rem}

A reasonable generalisation of Corollary \ref{cor:voronoi} would be to replace the cusp $\af=\infty$ on the left-hand side with an arbitrary cusp. One obtains the following extension of our theorem given in Corollary \ref{cor:cl_vor_1}.
We end this note by reproducing two of the most classical Vorono\"i formula from Corollary~\ref{cor:voronoi}. These agree with the formula given in \cite{KMV02}, for instance.

\begin{cor}\label{cor:cl_vor_1}
If $(a,N)=1$ then
\begin{multline}
	\sum_{\n\in\Nat} e\left(n\frac{a}{b}\right)a_f(n) n^{-\frac{k-1}{2}} F\left(n\right) = \\
	  \chi(b)\frac{\eta(f)}{b\sqrt{N}} \sum_{ n\in\Z_{\neq 0}}e\left(-n\frac{\overline{aN}}{b}\right)a_{\tilde{f}}(n)n^{-\frac{k-1}{2}}  [\mathcal{H}_f F]\left(\frac{n}{b^2N}\right) \nonumber
\end{multline}
where $i^k\eta(f)$ is the root number of the $L$-function associated to $f$ and $\tilde{f}$ is the dual newform determined by the functional equation $\Lambda(f,s)=i^k\eta(f)\Lambda(\tilde{f},1-s)$ of the completed $L$-function.
\end{cor}
\begin{proof}
Our corollary provides us with the formulae
\begin{multline}
	\sum_{\n\in\Nat} e\left(n\frac{a}{b}\right)a_f(n) n^{-\frac{k-1}{2}} F\left(n\right) \\ = b^{-1} \sum_{n\in\Z_{\neq 0}}e\left(-n\frac{\overline{a}}{bN}\right)a_f(n;{a}/{b})\left(\frac{n}{N}\right)^{-\frac{k-1}{2}} [\mathcal{H}_f F]\left(\frac{n }{b^2N}\right). \nonumber
\end{multline}
Without loss of generality we can assume that $(a,N)=1$ and use Proposition~\ref{prop:the_proposition} to obtain
\begin{multline}
	a_f(n;{a}/{b}) =  \Omega_{\chi,b,\delta(\af)}  a_{f}(n_{0})\,e\left(\dfrac{nd\overline{b}}{N}\right)\left( \frac{(n,N^{\infty})}{N}\right)^{\frac{k}{2}}\\
		\times\,\prod_{p\mid N}\omega_{\chi,p}({n}{p^{-n_p}})W_{p}(g_{n_{p}-d_{\pi_{p}}(b_{p}),0,u_{p}})
\end{multline}
Note that 
\begin{equation}
	e\left(-n\frac{\overline{a}}{bN}\right)e\left(n\frac{\overline{ab}}{N}\right) = e\left(-n\frac{\overline{aN}}{b}\right) \nonumber
\end{equation}
and
\begin{equation}
	\Omega_{\chi,b,\delta(\af)} = \prod_{p\mid N}\omega_{\chi,p}^{-1}\left(\frac{bN}{(N,p^{\infty})}\right) = \chi(b). \nonumber 
\end{equation}
The statement then follows by using Lemma~\ref{lm:first_atkin_lehner} and using the additional characters to twist $a_f(n_0)$ to $a_{\tilde{f}}(n_0)$.
\end{proof}
Of course this corollary exploits that if $(a,N)=1$ then the cusp $\mathfrak{b}= {a}/{b}$ is equivalent to $0$, which switches to $\infty$ under the classical Atkin--Lehner involution.
Lastly we consider the other extreme case $N\mid b$. In this case $\mathfrak{b}={a}/{b}$ is in fact equivalent to $\infty$.
\begin{cor}\label{cor:cl_vor_2}
Suppose $N\mid b$. Then
\begin{multline}
	\sum_{\n\in\Nat} e\left(n\frac{a}{b}\right)a_f(n) n^{-\frac{k-1}{2}} F\left(n\right) = \frac{\overline{\chi(a)}}{b} \sum_{n\in\Z_{\neq 0}}e\left(-n\frac{\overline{a}}{b}\right)a_{f}(n)n^{-\frac{k-1}{2}}  [\mathcal{H}_f F]\left(\frac{n }{b^2}\right). \nonumber
\end{multline}
\end{cor}  
Once again, this result would follow from the application of Corollary~\ref{cor:voronoi} together with Proposition~\ref{prop:the_proposition} and Lemma~\ref{lm:second_AL}, but this would be like shooting pigeons with canons. We give a much simpler proof here.
\begin{proof}
In this situation $\delta(\af)=1$ so that by Corollary~\ref{cor:voronoi} we have that 
\begin{multline}
	\sum_{\n\in\Nat} e\left(n\frac{a}{b}\right)a_f(n) n^{-\frac{k-1}{2}} F\left(n\right) = b^{-1} \sum_{ n\in\Z_{\neq 0}}e\left(-n\frac{\overline{a}}{b}\right)a_f(n;\af)n^{-\frac{k-1}{2}}  [\mathcal{H}_f F]\left(\frac{n}{b^2}\right) . \nonumber
\end{multline}
Furthermore, because $N\mid b$ the scaling matrix for $\af$ is in $\Gamma_0(N)$. Thus according to \eqref{eq:classical_change_of_scaling} we find that $a_f(n,\af)= \chi(a)^{-1}a_f(n)$. This concludes the proof.
\end{proof}

\bibliographystyle{amsplain}			
\bibliography{bibliography-voronoi-add-XR}	

\providecommand{\bysame}{\leavevmode\hbox to3em{\hrulefill}\thinspace}
\providecommand{\MR}{\relax\ifhmode\unskip\space\fi MR }
\providecommand{\MRhref}[2]{%
  \href{http://www.ams.org/mathscinet-getitem?mr=#1}{#2}
}
\providecommand{\href}[2]{#2}
\begin{thebibliography}{10}

\bibitem{As17}
E.~{Assing}, \emph{On the size of $p$-adic whittaker functions}, Trans. Amer.
  Math. Soc. (2018), Online.

\bibitem{As18}
\bysame, \emph{Yet another {$GL_2$} subconvexity result}, ArXiv e-prints
  (2018), Online.

\bibitem{Blomer-Brumley}
V.~Blomer and F.~Brumley, \emph{On the {R}amanujan conjecture over number
  fields}, Ann. of Math. (2) \textbf{174} (2011), no.~1, 581--605.

\bibitem{bump}
D.~Bump, \emph{Automorphic forms and representations}, Cambridge Studies in
  Advanced Mathematics, vol.~55, Cambridge University Press, Cambridge, 1997.

\bibitem{casselman}
W.~Casselman, \emph{On some results of {A}tkin and {L}ehner}, Math. Ann.
  \textbf{201} (1973), 301--314.

\bibitem{cogdell_BS}
J.~Cogdell, \emph{Bessel functions for {$GL_2$}}, Indian J. Pure Appl. Math.
  \textbf{45} (2014), no.~5, 557--582.

\bibitem{corbett-saha}
A.~{Corbett} and A.~{Saha}, \emph{On the order of vanishing of newforms at
  cusps}, Math. Res. Lett. \textbf{25} (2018), no.~6, 171--1804.

\bibitem{gelbart}
S.~Gelbart, \emph{Automorphic forms on ad\`ele groups}, Princeton University
  Press, Princeton, N.J. and University of Tokyo Press, Tokyo, 1975, Annals of
  Mathematics Studies, No. 83.

\bibitem{goldfeld-hundley-1}
D.~Goldfeld and J.~Hundley, \emph{Automorphic representations and
  {$L$}-functions for the general linear group. {V}olume {I}}, Cambridge
  Studies in Advanced Mathematics, vol. 129, Cambridge University Press,
  Cambridge, 2011, With exercises and a preface by Xander Faber.

\bibitem{GR07}
I.~Gradshteyn and I.~Ryzhik, \emph{Table of integrals, series, and products},
  seventh ed., Elsevier/Academic Press, Amsterdam, 2007.

\bibitem{jacquet-langlands}
H.~Jacquet and R.~Langlands, \emph{Automorphic forms on
  {$\operatorname{GL}(2)$}}, Lecture Notes in Mathematics, Vol. 114,
  Springer-Verlag, Berlin-New York, 1970.

\bibitem{knightly-li}
A.~Knightly and C.~Li, \emph{Traces of {H}ecke operators}, Mathematical Surveys
  and Monographs, vol. 133, American Mathematical Society, Providence, RI,
  2006.

\bibitem{KMV02}
E.~Kowalski, P.~Michel, and J.~VanderKam, \emph{Rankin-{S}elberg
  {$L$}-functions in the level aspect}, Duke Math. J. \textbf{114} (2002),
  no.~1, 123--191.

\bibitem{nps}
P.~Nelson, A.~Pitale, and A.~Saha, \emph{Bounds for {R}ankin--{S}elberg
  integrals and quantum unique ergodicity for powerful levels}, J. Amer. Math.
  Soc. \textbf{27} (2014), no.~1, 147--191.

\bibitem{ramakrishnan}
D.~Ramakrishnan and R.~Valenza, \emph{Fourier analysis on number fields},
  Graduate Texts in Mathematics, vol. 186, Springer-Verlag, New York, 1999.

\bibitem{Reznikov}
A.~Reznikov, \emph{Rankin-{S}elberg without unfolding and bounds for spherical
  {F}ourier coefficients of {M}aass forms}, J. Amer. Math. Soc. \textbf{21}
  (2008), no.~2, 439--477.

\bibitem{saha-hybrid}
A.~Saha, \emph{Hybrid sup-norm bounds for {M}aass newforms of powerful level},
  Algebra Number Theory \textbf{11} (2017), no.~5, 1009--1045.

\bibitem{saha-large-values}
Abhishek Saha, \emph{Large values of newforms on {$\operatorname{GL}(2)$} with
  highly ramified central character}, Int. Math. Res. Not. IMRN (2016), no.~13,
  4103--4131.

\bibitem{ralf-newforms}
R.~Schmidt, \emph{Some remarks on local newforms for {$\operatorname{GL}(2)$}},
  J. Ramanujan Math. Soc. \textbf{17} (2002), no.~2, 115--147.

\bibitem{Te_vor}
N.~Templier, \emph{Vorono{\"i} summation for {$\operatorname{GL}(2)$}},
  Representation Theory, Automorphic Forms, and Complex Geometry, Proceedings
  Volume in Honor of Wilfried Schmid, International Press (To appear), Online.

\end{thebibliography}
\end{document}